\documentclass[microtype]{gtpart}
\usepackage{amscd}
\usepackage{longtable}

\title[Solvable Groups and Free Divisors]
{Solvable Groups, Free Divisors and Nonisolated Matrix Singularities II:  
Vanishing Topology}

\author[J Damon]{James Damon}
\givenname{James}
\surname{Damon}
\address{Department of Mathematics \\
University of North Carolina \\\newline
Chapel Hill, NC 27599-3250  \\
USA}
\email{jndamon@math.unc.edu}
\urladdr{http://www.unc.edu/math/Faculty/jndamon/}

\author[B Pike]{Brian Pike} 
\givenname{Brian}
\surname{Pike}
\address{Department of Computer and Mathematical Sciences \\
University of Toronto Scarborough \\\newline
1265 Military Trail  \\
Toronto, ON M1C 1A4  \\
Canada}
\email{bpike@utsc.utoronto.ca}
\urladdr{http://www.utsc.utoronto.ca/~bpike/}



\arxivreference{1201.1579}
\arxivpassword{}

\volumenumber{}
\issuenumber{}
\publicationyear{}
\papernumber{}
\startpage{}
\endpage{}
\doi{}
\MR{}
\Zbl{}
\received{}
\revised{}
\accepted{}
\published{}
\publishedonline{}
\proposed{}
\seconded{}
\corresponding{}
\editor{}
\version{}

\theoremstyle{plain}
\newtheorem{Thm}{Theorem}[section]
\newtheorem{metatheorem}{Metatheorem}[section]
\newtheorem{Corollary}{Corollary}[section]
\newtheorem{Proposition}{Proposition}[section]
\newtheorem{Lemma}{Lemma}[section]

\theoremstyle{definition}
\newtheorem{Definition}{Definition}[section]
\newtheorem*{notation}{Notation}
\newtheorem{Example}{Example}[section]
\newtheorem{Remark}{Remark}[section]

\makeatletter
  \let\c@metatheorem=\c@Thm
  \let\c@Corollary=\c@Thm
  \let\c@Proposition=\c@Thm
  \let\c@Lemma=\c@Thm
  \let\c@Definition=\c@Thm
  \let\c@Example=\c@Thm
  \let\c@Remark=\c@Thm
  \let\c@SimpRemark=\c@Thm
\makeatother

\makeautorefname{Thm}{Theorem}
\makeautorefname{metatheorem}{Metatheorem}
\makeautorefname{Corollary}{Corollary}
\makeautorefname{Proposition}{Proposition}
\makeautorefname{Lemma}{Lemma}

\makeautorefname{Definition}{Definition}
\makeautorefname{Example}{Example}
\makeautorefname{Remark}{Remark}
\makeautorefname{SimpRemark}{Remark}

\numberwithin{equation}{section}

\newcommand{\codim}{{\operatorname{codim}}}
\newcommand{\sing}{{\operatorname{sing}}}

\newcommand{\wt}{{\operatorname{wt}}}
\newcommand{\rank}{{\operatorname{rank}}}
\newcommand{\dlog}{{\operatorname{Derlog}}}

\newcommand{\pd}[2]{\dfrac{\partial#1}{\partial#2}}

\def \bj {\mathbf {j}}

\def \b1 {\mathbf {1}}

\def \iti {\text{\it i}}

\def \itm {\text{\it m}}

\def \cA {\mathcal{A}}

\def \cD {\mathcal{D}}
\def \cE {\mathcal{E}}

\def \cK {\mathcal{K}}
\def \cL {\mathcal{L}}

\def \cO {\mathcal{O}}

\def \cQ {\mathcal{Q}}

\def \cV {\mathcal{V}}
\def \cW {\mathcal{W}}

\def \cY {\mathcal{Y}}

\def \gb {\beta}
\def \gg {\gamma}
\def \gd {\delta}

\def \gevar {\varepsilon}

\def \gl {\lambda}
\def \gs {\sigma}

\def\Sym{\mathrm{Sym}}
\def\Sk{\mathrm{Sk}}
\def\GL{\mathrm{GL}}
\def\Diff{\mathrm{Diff}}

\def\Pf{\mathrm{Pf}}

\def\g{\mathfrak{g}}

\def\Jac{\mathrm{Jac}}
\def\Im{\mathrm{Im}}
\def\rmdef{\mathrm{def}}
\def\rmlog{\mathrm{log}}

\begin{document}
\begin{abstract}
In this paper we use the results from the first part to compute the 
vanishing topology for matrix singularities based on certain spaces of 
matrices.  We place the variety of singular matrices in a geometric 
configuration of free divisors which are the ``exceptional orbit 
varieties'' for representations of solvable groups.  Because there are 
towers of representations for towers of solvable groups, the free divisors 
actually form a tower of free divisors $\cE_n$, and we give an inductive 
procedure for computing the vanishing topology of the matrix 
singularities.  The inductive procedure we use is an extension of that 
introduced by L\^{e}--Greuel for computing the Milnor number of an ICIS.  
Instead of linear subspaces, we use free divisors arising from the 
geometric configuration and which correspond to subgroups of the 
solvable groups.

Here the vanishing topology involves a singular version of the Milnor fiber; 
however, it still has the good connectivity properties and is homotopy 
equivalent to a bouquet of spheres, whose number is called the singular 
Milnor number. We give formulas for this singular Milnor number in terms 
of singular Milnor numbers of various free divisors on smooth subspaces, 
which can be computed as lengths of determinantal modules.  In addition 
to being applied to symmetric, general and skew-symmetric matrix 
singularities, the results are also applied to Cohen--Macaulay singularities 
defined as $2 \times 3$ matrix singularities.  We compute the Milnor 
number of isolated Cohen--Macaulay surface singularities of this type in 
$\C^4$ and the difference of Betti numbers of Milnor fibers for isolated 
Cohen--Macaulay 3--fold singularities of this type in $\C^5$.  
\end{abstract}

\maketitle

%
\section*{Introduction}  
\label{S:sec0} 
In this paper we make use of the results from the first part of the paper 
\cite{DP1} to introduce a method for computing the ``vanishing 
topology'' of nonisolated complex matrix singularities.  A complex 
matrix singularity arises from a holomorphic germ $f_0 \co \C^n, 0 \to M, 0$, 
where $M$ denotes the space of $m \times m$ complex matrices, which 
may be either symmetric or skew-symmetric (and then $m$ is even), or 
more general $m \times p$ complex matrices.  If $\cV$ denotes the 
``determinantal variety'' of singular matrices, then $\cV_0 = f_0^{-
1}(\cV)$ is the corresponding matrix singularity.  We shall also refer to 
the mapping $f_0$ as defining a matrix singularity; it can also be viewed 
as a ``nonlinear section of $\cV$'' (although we also allow $n 
\geq \dim(M)$).  In part I, we indicated many examples of matrix 
singularities for the classification of various types of singularities.   

For $m \times m$ matrices, if $n \leq \codim(\sing(\cV))$ and $f_0$ is 
transverse to $\cV$ off the origin, then $\cV_0$ has an isolated 
singularity, defined by $H \circ f_0$, where $H \co M \to \C$ denotes the 
determinant, or the Pfaffian in the skew-symmetric case ($m$ even).  
Using algebraic resolutions, Goryunov--Mond \cite{GM} showed that for 
isolated matrix singularities in all three cases, the Milnor number equaled 
$\tau$, which is a $\cK_H$--deformation theoretic codimension, with a 
correction term given by a two term Euler characteristic for an 
appropriate Tor complex.  
\begin{equation*}
\mu(H \circ f_0) \,\, = \,\, \tau  + (\gb_0 - \gb_1) \, . 
\end{equation*}
This explained an observed result of Bruce \cite{Br} for simple symmetric 
matrix singularities for $n = 2 = \codim(\sing(\cV)) - 1$.
   
Although the Milnor number in the isolated case can be computed from 
Milnor's formula, the relation between it and the deformation theoretic 
codimension suggests there may exist such a relation in the nonisolated 
case, where there are no known general results about the topology of the 
Milnor fiber.  However, the difficulty in determining the vanishing 
topology of matrix singularities in general is due to their highly singular 
structure.  Hence,  by the Kato--Matsumoto Theorem, its Milnor fiber will 
have very low connectivity and can have homology in many dimensions.  

We overcome this problem by viewing $f_0 \co \C^n, 0 \to M, 0$ as a 
nonlinear section of $\cV$ and consider instead the ``singular Milnor 
fiber''.  It is obtained as a ``stabilization of $f_0$'' and is 
homotopy equivalent to a bouquet of spheres of real dimension $n - 1$.  
The number of such spheres $\mu_{\cV}(f_0)$ is called the ``singular 
Milnor number'' of $f_0$, and it can be computed for free divisors 
$\cV$ (in the sense of Saito \cite{Sa}) by a Milnor-type formula as the 
length of of a determinantal module, \cite{DM} and \cite{D2}.    
In the case when $n < \dim(\sing(\cV))$, then $\cV_0$ is an isolated 
singularity and these are the usual Milnor fiber and Milnor number.  
That matrix singularities $\cV$ are essentially never free divisors 
explains the need for a correction term in \cite{GM} for the isolated case.  

Instead we shall introduce an inductive method which extends that 
introduced by  L\^{e}--Greuel \cite{LGr} for computing the Milnor number of 
an ICIS.  Their method uses a geometric configuration formed from a 
flag of linear subspaces transverse to the map germ which we replace 
with a tower of linear free divisors constructed in Part I \cite{DP1}.  
These arise from a tower of (modified) Cholesky-type representations of 
solvable linear algebraic groups.  This allows us to adjoin a linear free 
divisor to the determinantal variety $\cV$ to obtain another linear free 
divisor, providing a ``free completion'' of $\cV$.  

The general form of the formula which we give expresses  
$\mu_{\cV}(f_0)$ as a linear combination with integer coefficients
\begin{equation}
\label{Eqn0.1}
 \mu_{\cV}(f_0) \,\,  = \,\,  \sum_{i} \, a_i \mu_{\cW_i}(f_0)
\end{equation}
where the $\cW_i$ are free divisors on linear subspaces of $M$.  Thus, we 
can express $\mu_{\cV}(f_0)$ as a linear combination of singular Milnor 
numbers, each of which can be computed using results from \cite{D2} as 
lengths of determinantal modules.
 
If we view these singular Milnor numbers as functions on the space of 
germs $f_0$ transverse to the varieties off $0$, then \eqref{Eqn0.1} can be 
written more simply as
\begin{equation}
\label{Eqn0.2}
 \mu_{\cV} \,\,  = \,\,  \sum_{i} \, a_i \mu_{\cW_i} \, .
\end{equation}

Furthermore, the method allows us to compute more generally the singular 
Milnor numbers for nonisolated matrix singularities on an ICIS $X$.  There 
is a metatheorem which states that if $X$ is defined by $\varphi \co \C^n, 0 
\to \C^p, 0$, and the formula \eqref{Eqn0.2} for $\mu_{\cV}$ is obtained by 
the inductive process then the process also yields the formula
\begin{equation}
\label{Eqn0.3}
 \mu_{\varphi, \cV} \,\,  = \,\,  \sum_{i} \, a_i \mu_{\varphi, \cW_i}
\end{equation}
where $\mu_{\varphi, \cV}(f_0)$, respectively $\mu_{\varphi, 
\cW_i}(f_0)$, are the singular Milnor numbers for $f_0 | X$ as nonlinear 
sections of $\cV$, resp. $\cW_i$, and can again be computed in terms of 
lengths of determinantal modules using a generalization of the 
L\^{e}--Greuel theorem given in \cite{D2}.  

These formulas are applied in \S \ref{S:symmatr}, \ref{S:sec6}, and 
\ref{S:sec8} to obtain explicit formulas for symmetric and general $2 
\times 2$ and $3 \times 3$ matrices, and $4 \times 4$ skew-symmetric 
matrices.

Furthermore, general $2 \times 3$ matrix singularities are not complete 
intersection singularities; however they are Cohen--Macaulay 
singularities by the Hilbert--Burch Theorem \cite{Hi,Bh}.  We next 
apply these methods in \S\ref{S:sec7} to obtain the singular vanishing 
Euler characteristic $\tilde \chi_{\cV}$ as a linear combination as in 
\eqref{Eqn0.2}. We then deduce a formula for the Milnor number of isolated 
$2\times 3$
Cohen--Macaulay surface singularities in $\C^4$ as an alternating sum of 
lengths of determinantal modules (\fullref{ThmCM}).  Furthermore, 
for isolated 3--fold $2\times 3$ Cohen--Macaulay singularities, we give an analogous 
formula for the difference between the second and third Betti numbers 
$b_3 - b_2$ of the Milnor fiber (\fullref{Thmb2-b3}).  This formula 
is also valid for $2\times 3$ Cohen--Macaulay singularities defined as matrix 
singularities defined on an ICIS. 
 
This formula has been programmed in Macaulay2 by the second author 
\cite{P2} and has been used to compute for the simple
isolated 
Cohen--Macaulay 
singularities, classified by Fr\"{u}hbis-Kr\"{u}ger--Neumer \cite{FN}, the 
Milnor numbers for those in $\C^4$ and the difference of Betti numbers 
for the Milnor fiber for the 3--fold singularities in $\C^5$.  In 
\S\ref{S:sec9}, these computer calculations are applied to verify a 
conjecture relating $\mu$ and $\tau$ for the surface case, and discover 
unexpected behavior of $b_3 - b_2$ and $\tau$ for the 
3--fold singularities.
 
Besides obtaining general formulas as in \eqref{Eqn0.2} for the various 
cases, we also introduce two methods of reduction.  In the case of $2 
\times 2$ symmetric matrices, the terms in the linear combination 
represent the lengths of determinantal modules and the algebraic 
relations between these modules then allow us to combine them into a 
``Jacobian formula''.  This is a first step to finding more 
general reduction formulas to simplify \eqref{Eqn0.2}. 
 
The second method of ``generic reduction'' can be applied to all 
cases and uses the ``defining codimensions'' of the $\cW_i$ in 
$M$.  We may rewrite \eqref{Eqn0.2} in the form 
\begin{equation}
\label{Eqn0.4}
 \mu_{\cV} \,\,  = \,\,  \gl_0  \, +  \, \gl_1  \, + \dots  +  \, \gl_{N-1} 
\qquad \qquad (N = \dim(M))  
\end{equation}
where $\gl_j$ denotes the sum of the terms in \eqref{Eqn0.2} for which 
the defining codimension of $\cW_i$ is $j$.  If $\codim (\Im(df_0(0))) = k$ 
and we may apply a generic matrix transformation to $f_0$ so that 
$\Im(df_0(0))$ projects submersively onto all of the defining linear 
subspaces of codimension $ \geq k$ associated to the $\cW_i$, then 
$\gl_i(f_0) = 0$ for $i \geq k$, and the formula \eqref{Eqn0.4} can be 
reduced to 
\begin{equation}
\label{Eqn0.5}
 \mu_{\cV}(f_0) \,\,  = \,\,  \gl_0(f_0)  \, +  \, \gl_1(f_0)  \, + \dots +  \, 
\gl_{k-1}(f_0) \, . 
\end{equation}
In essence the remaining terms are ``higher order terms'' which 
do not contribute in the generic case.  We deduce a number of 
consequences of this reduction for the different types of matrices, and 
obtain $\mu = \tau$ type results for generic corank $1$ mappings defining 
matrix singularities of the various types (\fullref{Thm9.2}). 
 
In this paper we have only derived the specific formulas for small 
matrices of various types.  These required an understanding of the roles of 
certain subgroups and block representations on subspaces and their 
relation with the intersection of orbits of the subgroups with the spaces 
of singular matrices.  To continue the analysis to more general matrices 
requires a more thorough analysis of such subgroups and their block 
representations on subspaces.  This work is ongoing.  Because the method 
applies quite generally to the exceptional orbit varieties for 
representations of solvable linear algebraic groups which form 
``block representations'' having associated 
``$H$--holonomic'' free divisors, these results will then as well 
extend to many other representations of solvable linear algebraic groups.  

\noindent\textbf{Acknowledgments}\stdspace
The authors would like to thank David Mond, Shrawan Kumar, and Jonathan 
Wahl for several very helpful conversations and the referee for the careful 
reading of the paper and his suggestions for a number of improvements to 
the paper.

The first author was 
partially supported by the Simons Foundation grant 230298, 
and the National Science Foundation grant DMS-1105470.
This paper contains work from the second author's Ph.D. dissertation
at the University of North Carolina at Chapel Hill.

\section{Outline of the method}  
\label{S:sec1a} 
We begin by outlining how we extend the L\^ e--Greuel method to apply to 
matrix singularities, and then illustrate the calculation for the simplest 
case of $2 \times 2$ symmetric matrices.   

Let $M$ be the space of $m \times m$ complex matrices which are 
symmetric or skew-symmetric, or $m \times p$ general matrices.   We 
also let $\cV$ denote the subvariety of singular matrices in $M$ (by which 
we mean more singular than the generic matrix in $M$). 
\begin{Definition} A matrix singularity is defined by a holomorphic germ 
\begin{equation}
f_0 \co \C^n, 0 \longrightarrow M, 0 
\end{equation}
(or more generally, $f_0 \co X, 0 \to M, 0$ for an analytic germ $X, 0$).  
The pull-back variety $\cV_0 = f_0^{-1}(\cV)$ is the \emph{matrix 
singularity} defined by $f_0$.
\begin{equation}
\label{tag1.1}
\begin{CD}
  @.  {\C^n,0} @>{f_0}>> M,0\\
@.  @AAA @AAA \\
{f_0^{-1}(\cV)}  @=  \cV_0, 0  @>>> \cV, 0  
\end{CD} 
\end{equation}
\end{Definition}

For these singularities we require that $f_0$ is transverse to $\cV$  off 
$0\in \C^n$ (i.e. to the canonical Whitney stratification of $\cV$). 
The determinantal varieties $\cV$ are highly singular.  The singular set of 
the determinantal varieties has codimension in $M$ equal to $3$ 
(symmetric case), $4$ general $m \times m$ case, or $6$ for the
skew-symmetric case ($m$ even); and by the Kato--Matsumoto Theorem 
\cite{KM}, the Milnor fiber of $\cV_0$ will only be guaranteed to be 
1--connected (symmetric case), 2--connected (general case), or 
4--connected (skew-symmetric case). 
 
To describe their vanishing topology, we initially replace the Milnor fiber 
by the ``singular Milnor fiber''.  As $f_0 \co \C^n, 0 \to M, 0$ is 
transverse to $\cV$ off $0$, we may use instead a stabilization
$f_t\co 
B_{\gevar} \to M$ of $f_0$.  This means that for $t \neq 0$, $f_t$ is 
transverse to $\cV$ on $B_{\gevar}$.  The \emph{singular Milnor fiber} is 
then the fiber $\cV_t = f_t^{-1}(\cV)$.  By results in \cite{DM} and 
\cite{D2} (using a result of L\^{e}), which are valid for any hypersurface 
$\cV$, the singular Milnor fiber $\cV_t$ is homotopy equivalent to a 
bouquet of spheres of real dimension $n - 1$, whose number we denote by 
$\mu_{\cV}(f_0)$ and which we call the ``singular Milnor 
number''.  If $\cV$ is instead a complete intersection, or if $f_0
\co X,0 \to M, 0$ 
for an ICIS $X, 0$, the singular Milnor fiber continues to be 
homotopy equivalent to a bouquet of spheres \cite{D2}.  If $\cV$ is not a 
complete intersection, the singular Milnor fiber need not be  homotopy
equivalent to a bouquet of spheres, so we consider instead the
\emph{singular vanishing Euler characteristic} $\tilde \chi_{\cV}(f_0) = 
\chi(\cV_t) - 1$.  
The singular Milnor numbers $\mu_{\cV}(f_0)$ have Milnor-type formulas 
if $\cV$ is a free divisor or a free divisor on a smooth subspace (see \S 
\ref{S:sec2}). 
 
However, in general the determinantal varieties consisting of singular 
matrices are not free divisors.  Consequently, we will proceed by 
modifying the method of L\^e--Greuel to compute them inductively using 
free divisors.   We recall how the L\^{e}--Greuel formula is used to 
compute the Milnor number of an ICIS.  
\subsection{Computing Milnor Numbers of ICIS via Geometric 
Configurations}
For an isolated hypersurface singularity defined by $f \co \C^n, 0 \to \C, 0$, 
the Milnor number is computed by Milnor's algebraic formula 
$$    \mu(f) \, \, = \,\, \dim_{\C}\!\left( \cO_{\C^n, 0}/ \Jac(f)\right),   $$
where $\Jac(f)$ is the ideal generated by the partials $\pd{f}{x_i}$, $i = 1, 
\dots , n$. 
By contrast, except in the weighted homogeneous case, there is no 
analogous Milnor-type formula for computing the Milnor number of an ICIS 
$f \co \C^n, 0 \to \C^p, 0$.  Instead, for a general ICIS, the  L\^{e}--Greuel 
formula provides an inductive method as follows.

We choose a geometric configuration which consists of a complete flag of 
subspaces $0 \subset \C  \subset \C^2  \subset \dots  \subset \C^p$ 
transverse to $f$ off $0$. If $(y_1, \dots , y_p)$ denote coordinates 
defining these subspaces, we let $\mu_{y_1, \dots , y_k}(f) = 
\mu (\pi_k \circ f)$, where $\pi_k$ denote projection onto the subspace 
$\C^k \times \{ 0\}$.  Then, the Milnor number $\mu(f)$ can be computed 
as an alternating sum
\begin{multline}
\label{Eqn4.1}
  \mu(f) \, \, = \,\, \left(\mu_{y_1, \dots , y_p}(f) + \mu_{y_1, \dots , y_{p-
1}}(f)\right)\, - \, \left(\mu_{y_1, \dots , y_{p-1}}(f) \, + \, \mu_{y_1, \dots , y_{p-
2}}(f)\right) \\
  \, + \dots \, \pm \left(\left((\mu_{y_1, y_{2}}(f) \, + \, \mu_{y_1}(f)\right) - 
\mu_{y_1}(f)\right) \, , 
\end{multline}
where each 2--term sum in parentheses represents the Milnor number of 
an isolated singularity on an ICIS and can be computed using the 
L\^{e}--Greuel Theorem (with $\mu_{y_1}(f)$ computed by Milnor's 
formula).
\begin{Thm}[L\^{e}--Greuel]
\label{ThmLeGr} 
For an ICIS $f = (f_1, f_2) \co \C^n, 0 \to \C^{k+1}, 0$, with $f_2
\co \C^n, 0 \to 
\C^k, 0$ also an ICIS, 
$$ \mu(f) + \mu(f_2) \quad = \quad  \dim_{\C}\!\left(\cO_n /(f_2^*\itm_k + 
\Jac(f))\right) \, .  $$  
where $\Jac(f)$ now denotes the ideal generated by the $(k+1) \times 
(k+1)$ minors of $df$. 
\end{Thm}
Thus, $\mu(f)$ is not computed directly, but rather as an alternating sum 
of lengths of algebras which are defined using a configuration of 
subspaces in $\C^p$.
\subsection{Inductive Procedure for Computing Singular Milnor 
Numbers via Free Completions}
We will use an analogous approach for computing the singular Milnor 
number of a matrix singularity.  We give an inductive approach, for which   
the geometric configuration is given by a free divisor $\cE_m$ appearing 
in one of the towers of free divisors from Part I \cite{DP1} (see 
\fullref{table3.0}).  This provides a ``free completion'' of the 
determinantal variety $\cD_m$ of singular matrices, $\cE_m = 
\pi^{*}\cE_{m-1} \cup  \cD_m$.  

Quite generally we define
\begin{Definition} 
\label{Def:freecompl}
A hypersurface singularity $\cW, 0 \subset \C^N, 0$ has a {\em free 
completion} if there is a free divisor $\cV, 0 \subset \C^N, 0$ such that 
$\cV \cup \cW, 0$ is again a free divisor.  
\end{Definition}
Then, we may apply \eqref{Eqn2.10} of \fullref{Lem2.9} to obtain
\begin{equation}
\label{Eqn1a.2}
  \mu_{\cD_m}(f_0) \,\, = \,\,  \mu_{\cE_m}(f_0) \, - \, 
\mu_{\pi^* \cE_{m-1}}(f_0) \, + \, (-1)^{n-1}\, \tilde \chi_{\pi^* \cE_{m-
1} \cap \cD_m}(f_0)  \, .
\end{equation}
In our situations, all of the $\pi^*\cE_m$ are $H$--holonomic (see 
beginning of \S \ref{S:sec2} and \S \ref{S:sec3}).  Thus, the 
$\mu_{\pi^*\cE_m}$ can be computed as lengths of determinantal modules 
by \fullref{ThmAFD}.  This reduces the calculation of 
$\mu_{\cD_m}(f_0)$ to computing $\tilde \chi_{\pi^* \cE_{m-1} \cap 
\cD_m}(f_0)$.  

We proceed inductively to decompose 
$\pi^* \cE_{m-1} \cap \cD_m$ into a union of components each of which 
can be represented as divisors on ICIS.  We then use either free 
completions for these divisors or completions by divisors which 
themselves have free completions.  We may again inductively apply 
\fullref{Lem2.9} to further reduce to computing the vanishing Euler 
characteristics for divisors on ICIS, where we repeat the inductive 
process.  Eventually we are reduced to computing the singular Milnor 
numbers of almost free divisors on ICIS, which we can compute using 
either \fullref{ThmAFD} or \fullref{ThmAFDICIS}.  

In analogy with the notation used to explain the case of ICIS, to represent 
the singular Milnor number of $f_0$ for a variety defined by $(g_1, \dots , 
g_r)$, we use the notation $\mu_{g_1, \dots , g_r}(f_0)$.  The final form 
the formula will take is that of \eqref{Eqn0.2}, where each $\mu_{\cW_i}$ 
is given in the form just described. 

If instead we consider matrix singularities $f_0 \co X, 0 \to M, 0$ on an ICIS 
$X, 0$ defined by $\varphi \co \C^n, 0 \to \C^p, 0$, then the same arguments 
may be repeated to obtain a formula of the form \eqref{Eqn0.3}. 
\subsection{$2 \times 2$ Symmetric Matrix Singularities}
As an initial example to illustrate these ideas, we consider the 
$2 \times 2$ symmetric matrices, denoted $\Sym_2$, and use coordinates $ 
\begin{pmatrix}
a  &b \\
b  & c 
\end{pmatrix}$.  The variety of singular matrices is $\cD_2^{sy}$ defined 
by $ac -b^2 = 0$.  Then, by Theorem 6.2 of \cite{DP1}, it has a free 
completion $\cE_2^{sy} = \pi^{*} \cE_1^{sy} \cup \cD_2^{sy}$, where 
$\cE_2^{sy}$ is defined by $a\left(ac -b^2\right) = 0$ and $\pi^{*}\cE_1^{sy}$ by $a 
= 0$. 
 
By the preceding, it is sufficient to determine $\tilde \chi_{\pi^* 
\cE_1^{sy} \cap \cD_2^{sy}}(f_0)$.  Then, set-theoretically, 
$$ \pi^* \cE_1^{sy} \cap \cD_2^{sy} \,\, = \,\, V\left(a, ac-b^2\right) \,\, = \,\, V(a, 
b)  \, . $$
Hence, 
$$ \tilde \chi_{\pi^* \cE_1^{sy} \cap \cD_2^{sy}} \,\, = \,\, (-1)^{n-2} 
\mu_{a,b} \, . $$
Since $\mu_{\pi^{*}\cE_1^{sy}}(f_0) = \mu_a(f_0)$, by substituting into 
\eqref{Eqn1a.2} 
we obtain
\begin{equation}
\label{Eqn1a.3}
  \mu_{\cD_2^{sy}}(f_0)   \,\, = \,\,   \mu_{\cE_2^{sy}}(f_0) \, -  \, 
\left(\mu_a(f_0) \,  + \, \mu_{a, b}(f_0)\right) 
\end{equation}
where $\mu_{\cE_2^{sy}}(f_0)$ can be computed via 
\fullref{ThmAFD} as the length of a determinantal module and 
$\mu_a(f_0) + \mu_{a, b}(f_0)$, by the L\^{e}--Greuel formula
(\fullref{ThmLeGr}).  A complete statement is given in 
\fullref{Thmsym2}.  

This example is especially simple as $\pi^* \cE_1^{sy} \cap \cD_2^{sy}$ is 
set-theoretically a complete intersection.  In general it will require a 
number of inductive steps to decompose $\pi^* \cE_{m-1} \cap \cD_m$ 
and use auxiliary solvable group representations to construct additional 
free completions for the components.  
\begin{Remark}
\label{Rem1a.1}
In order to apply the inductive method, we must have the germ $f_0 \co \C^n, 
0 \to M, 0$ transverse off $0$ to each of the free divisors on the 
subspaces  and their intersections.  We use the terminology that $f_0$ is 
\emph{transverse to the associated varieties} to indicate that it is 
transverse to all of these associated free divisors and their intersections.  

For matrix singularities, we only assume initially that $f_0$ is 
transverse off $0$ to the determinantal variety $\cD$.  To ensure that 
$f_0$ is also transverse to the associated varieties, we may apply to 
$f_0$ an element of the larger groups $\GL_m$ or $\GL_m \times \GL_p$ 
which preserve the determinantal variety of singular matrices.  The 
actions of the groups $\GL_m$ or $\GL_m \times \GL_p$ are transitive on 
the strata of the determinantal variety $\cD$ (by the classification of 
complex bilinear forms and echelon form for linear transformations).  The 
complement of $\cD$ consists of matrices of maximal rank, and again by 
the classification, they belong to a single orbit of these groups.  Hence, by 
the parametrized transversality theorem, for almost all elements $g$ of 
the appropriate group, the composition of the action of $g$ with $f_0$, 
denoted $g\cdot f_0$, is transverse to the associated varieties. Hence, 
these will preserve $\cD$ and move $f_0$ into general position off $0$ 
relative to the associated varieties.  
\end{Remark}
There are three essential ingredients which allow the general 
computations to be carried out for the various matrix types in the later 
sections:  
\begin{itemize}
\item  First, the singular Milnor numbers are computed in terms of a 
certain deformation theoretic codimension for $\cK_{H}$--equivalence.  In 
\S \ref{S:sec1} we relate this to the equivalence $\cK_{M}$ for matrix 
singularities and a related equivalence $\cK_{\cV}$ for viewing germs as 
nonlinear sections of the variety $\cV$ of singular matrices.  We also 
recall the formulas for codimensions as lengths of modules.  
\item  Second, we recall in \S \ref{S:sec2} the formulas for computing the 
singular Milnor numbers and formulas involving them and singular 
vanishing Euler characteristics. 
\item  Third, in \S \ref{S:sec3} we summarize the results from part I 
which construct the towers of free divisors and certain auxiliary free 
divisors needed for the various types of matrix singularities.
\end{itemize} 
\section{Equivalence Groups for Matrix Singularities}  
\label{S:sec1} 
There are several different equivalences that we shall consider for 
matrix singularities $f_0 \co \C^n, 0 \to M, 0$ with $\cV$ denoting the 
subvariety of singular matrices in $M$.  The one used in classifications is 
\emph{$\cK_M$--equivalence}:  We suppose that we are given an action of a 
group of matrices $G$ on $M$.  For symmetric or skew-symmetric 
matrices, it is the action of $\GL_m(\C)$ by $B\cdot A =  B\, A\, B^T$.  For 
general $m \times p$ matrices, it is the action of $\GL_m(\C) \times 
\GL_p(\C)$ by $(B, C)\cdot A =  B\, A\, 
C^{-1}$.  Given such an action, then the group $\cK_M$ consists of pairs 
$(\varphi, B)$, with $\varphi$ a germ of a diffeomorphism of $\C^n, 0$ and 
$B$ a holomorphic germ $\C^n, 0 \to G, I$.  The action is given by
$$    f_0(x)  \mapsto  f_1(x)\,  = \, B(x)\cdot\left( f_0\circ \varphi^{-1}(x)\right) \, . 
$$
For one space $M$ and group $G$, we use the generic notation 
$\cK_M$ for any of these groups of equivalence (Gervais had earlier 
considered this type of equivalence, referring to it as $G$--equivalence 
\cite{Ge1,Ge2}).  

In addition to $\cK_M$, there are two other commonly used groups.  
\subsection{$\cK_{\cV}$ and $\cK_{H}$--equivalence for Matrix 
Singularities}
If we view $f_0$ as a ``nonlinear section of $\cV$'' (even for a 
more general germ $\cV, 0$), $\cK_{\cV}$--equivalence is defined by the 
actions of pairs of diffeomorphisms $(\Phi, \varphi)$, preserving $\C^n 
\times \cV$ (see \cite{D1}).  
\begin{equation}
\label{tag1.2}
\begin{CD}
 {\C^n \times \C^N, 0} @>{\Phi}>> {\C^n \times \C^N, 0} @<{\iti}<<  {\C^n 
\times \cV, 0}\\
@V{\pi}VV  @V{\pi}VV  @. \\
 \C^n, 0  @>{\varphi}>> \C^n, 0   @.  
\end{CD} 
\end{equation}
For $\cV_0 = f_0^{-1}(\cV)$, it gives an ambient equivalence of $\cV_0, 0 
\subset \C^n,0$.  

There is a third equivalence, $\cK_H$--equivalence, introduced in 
\cite{DM}, which requires moreover that $\Phi$ given above preserves all 
of the level sets of $H$.  Here $H$ is chosen to be a ``good defining 
equation'' for $\cV$, which means there is an ``Euler-like 
vector field'' $\eta$ such that $\eta(H) = H$.  In the weighted 
homogeneous case such as for determinantal varieties, we use the Euler 
vector field (for general $\cV$ we may always replace $\cV$ by $\cV 
\times \C$ and $\pd{ }{t}$ is such a vector field for the defining equation 
$e^t\cdot H$). 
 
All of these equivalence groups have corresponding unfolding groups and 
belong to the class of geometric subgroups of $\cA$ or $\cK$, so all of the 
basic theorems of singularity theory in the Thom--Mather sense are valid 
for them (see \cite{D1,D3,D6}).  In particular, germs 
which have finite codimension for one of these groups have versal 
unfoldings, and the deformation theoretic spaces for these groups play an 
important role. 
 
We let $\theta_N$ denote the module of germs of vector fields on $\C^N, 
0$, and $I(\cV)$ the ideal of germs vanishing on $\cV$, and define, after 
Saito \cite{Sa} the module of \emph{logarithmic vector fields}
\begin{equation*}
  \dlog (\cV) \quad  = \quad  \{ \zeta \in \theta_N : \zeta (I(\cV)) \subseteq I(\cV) \}. 
\end{equation*}
For good defining equation $H$, we also define 
\begin{equation*}
  \dlog (H) \quad  = \quad  \{ \zeta \in \theta_N : \zeta (H) = 0 \}. 
\end{equation*}
If $H$ is a good defining equation,
\begin{equation*}
  \dlog(\cV)  \quad =  \quad  \dlog(H) \, \oplus \, \cO_{\C^N, 0}\{ \eta \} \, 
.
\end{equation*}
These modules both appear in infinitesimal calculations for the groups.  

If $\dlog(\cV)$ is generated by $\zeta_0, \dots , \zeta_r $, then the 
extended tangent space is given by
\begin{equation}
\label{Eqn1.4}
T\cK_{\cV, e} \cdot f_0 \,\quad = \, \quad \cO_{\C^n, 0} 
\left\{ \pd{f_0}{x_1}, \dots , \pd{f_0}{x_n}, \zeta_0 \circ f_0, \dots , \zeta_r
 \circ f_0 \right\} \, .
\end{equation} 
The analog of the deformation tangent space $T^1$ is the extended 
$\cK_{\cV}$  normal space 
\begin{equation*}
N\cK_{\cV, e}\cdot f_0 \quad = \quad \theta(f_0) / T\cK_{V, e}\cdot f_0 
\quad \simeq  \quad \cO_{\C^n, 0}^{(p)}/ T\cK_{V, e}\cdot f_0
\end{equation*}  
where as usual $\theta(f_0)$, the module of germs of holomorphic vector 
fields along $f_0$, is the free $\cO_{\C^n, 0}$ module generated by 
$\left\{ \pd{ }{x_i}\right\}$,  $1 \leq i \leq n$.  Likewise, if $\zeta_0$ denotes the 
Euler-like vector field with the remaining $\zeta_i$ generating $\dlog 
(H)$, then $T\cK_{H, e}$ is obtained by deleting  $\zeta_0 \circ f_0$ in 
\eqref{Eqn1.4}, with $N\cK_{H, e}$ denoting the corresponding quotient.  As 
usual, the dimensions of these extended normal spaces are the extended 
codimensions $\cK_{\cV, e}$--$\codim (f_0)$, resp. $\cK_{H, e}$--$\codim 
(f_0)$. 
 
There is a direct relation between these groups and $\cK_M$.  The 
extended tangent space for $\cK_M$ is obtained by an analogous formula 
to \eqref{Eqn1.4} except the generators of $\dlog(\cV)$ are replaced by 
vector fields for the matrix equivalence group $G$ acting on $M \simeq 
\C^N$.  They are of the form $\xi_{v_i}(x) = \pd{ }{t}(\exp(tv_i)\cdot x)_{| 
t = 0}$, for $\{ v_i\}$ a basis for the Lie algebra $\g$ of $G$.  In the 
terminology of part I, we refer to these as the ``representation 
vector fields''.  

The reason these are so closely related for matrix singularities is due to a 
collection of results due to J\'{o}zefiak \cite{J}, J\'{o}zefiak--Pragacz 
\cite{JP}, and Gulliksen--Neg\r{a}rd\cite{GN}.  Goryunov--Mond \cite{GM} 
recognized that these results prove that for the three types of $m \times 
m$ matrices (symmetric, skew-symmetric (with $m$ even), or general 
matrices) that the modules of vector fields generated by the 
representation vector fields are exactly $\dlog(\cV)$, for $\cV$ the 
determinantal variety of singular matrices.  It then follows that $\cK_M$ 
and $\cK_{\cV}$ have the same tangent spaces; and when using the 
standard methods for studying equivalence of singularities, they give the 
same equivalence. 
 
In addition, as noted in \cite{DM}, if $f_0$ is weighted homogeneous for 
the same set of weights as $\cV$, then the extended tangent spaces of 
$f_0$ for $\cK_{\cV}$ and $\cK_H$ are the same.  Hence,
\begin{equation}
\label{Eqn1.5}
\cK_{M, e}\textrm{--}\codim (f_0) \,\, = \,\, 
\cK_{\cV, e}\textrm{--}\codim (f_0)\,\, = \,\, 
\cK_{H, e}\textrm{--}\codim (f_0)  \, .
\end{equation}

Thus, Bruce's observed result \cite{Br} about simple symmetric matrix 
singularities and the result of Goryunov--Mond \cite{GM} both concern the 
relation between the Milnor number $\mu(H \circ f_0)$ and 
$\cK_{H, e}$--$\codim (f_0)$.  We next consider how this relates to the 
case of nonisolated  matrix singularities.  
\section{Singular Milnor Fibers and Singular Milnor Numbers}  
\label{S:sec2} 
The singular Milnor numbers can  be explicitly computed in the case $\cV$ 
is a \emph{free divisor}.  This term was introduced by Saito \cite{Sa} for 
hypersurface germs $\cV, 0 \subset \C^N, 0$ for which $\dlog (\cV)$ is a 
free 
$\cO_{\C^N}$--module, necessarily of rank $N$.  In this case, if $f_0
\co \C^n 
\to M, 0$ is transverse to $\cV$ off $0$ ($\in \C^n$), we refer to $\cV_0 = 
f_0^{-1}(\cV)$ as an \emph{almost free divisor (AFD)}. 
 
A free divisor $\cV$ is called \emph{holonomic} by Saito if at any point $z 
\in \cV$ the generators of $\dlog (V)$ evaluated at $z$ span the tangent 
space of the stratum containing $z$ of the canonical Whitney 
stratification of $\cV$.  If this still holds true using $\dlog(H)$ instead 
then we say it is 
\emph{$H$--holonomic} \cite{D2}.
 
Then, the results in \cite[Thm 5]{DM} (for locally weighted homogeneous 
free divisors) and \cite[Thm 4.1]{D2} (extended to $H$--holonomic free 
divisors) combine to give the following formula for the singular Milnor 
number.  
\begin{Thm}
\label{ThmAFD} 
If $\cV \subset \C^N$ is an $H$--holonomic free divisor, and $f_0 \co \C^n, 0 \to 
\C^N, 0$ is transverse to $\cV$ off $0$, then
\begin{equation}
\label{Eqn2.1}
       \mu_{\cV}(f_0) \quad = \quad \cK_{H, e}\makebox{--}\codim(f_0)  
\end{equation}
where the RHS is computed as the length of a determinantal module. 
\end{Thm}
\begin{Remark}
\label{Rem2.1} 
We note by \cite[Lemma 2.10]{D2} that as $\cV$ is $H$--holonomic, $f_0$ 
is transverse to $\cV$ off $0$ if and only if $f_0$ has finite 
$\cK_{H, e}$--codimension.  
\end{Remark}
\subsection{Almost Free Divisor (AFD) on an ICIS}
This formula further extends to the case $f_0 \co X, 0 \to \C^N, 0$ where $X, 
0 \subset \C^n,0$ is an ICIS defined by $\varphi \co \C^n, 0 \to \C^p, 0$.  
In our situation, we consider the case where $f_0 | X$ is transverse to a 
$H$--holonomic free divisor $\cV$ off $0$.  Then, as in \S \ref{S:sec1a}, we 
consider a stabilization $f_t\co B_{\gevar} \to M$ of $f_0$, for which $f_t | 
X \cap B_{\gevar}$ is transverse to $\cV$ for $t \neq 0$.  For $\cV_t = 
f_t^{-1}(\cV)$, $\cV_t \cap X \cap B_{\gevar}$ is homotopy equivalent to 
a bouquet of spheres of real dimension $n - p -1$ 
\cite[\S 7]{D2}.  We denote by $\mu_{\varphi, \cV}(f_0)$ the number of 
such spheres and refer to this number as the \emph{singular Milnor number 
of} $f_0 | X$.  Then, the singular Milnor number can be computed by the 
following generalization of the L\^{e}--Greuel formula, see \cite[\S 9]{D2} 
or \cite[\S 4]{D3}.  
\begin{Thm}[AFD on an ICIS]
\label{ThmAFDICIS}  
Let $\cV, 0 \subset \C^N, 0$ be an $H$--holonomic free divisor as above.  
Suppose $X, 0 \subset \C^n,0$ is an ICIS defined by $\varphi \co \C^n, 0 \to 
\C^p, 0$, and that $f_0 | X$ is transverse to $\cV$ off $0$.  Let $F = 
(\varphi, f_0) \co \C^n, 0 \to \C^{p +N}, 0$.  Then, 
\begin{equation}
\label{Eqn2.2}
       \mu_{\varphi, \cV}(f_0)  + \mu(\varphi) = 
\dim_{\C}\! \left( \cO_{X, 0}^{p +N} \middle/ \cO_{X, 0}\!\left\{ \pd{F}{x_1}, \dots , 
\pd{F}{x_n}, \zeta_1 \circ f_0, \dots , \zeta_{N-1} \circ f_0 \right\}\right),
\end{equation}
where $\dlog(H)$ is generated by $\zeta_i$, $i = 1, \dots , N-1$.
\end{Thm}
With $\mu(\varphi)$ computed by the L\^{e}--Greuel formula, \eqref{Eqn2.2} 
then yields the singular Milnor number $\mu_{\varphi, \cV}(f_0)$.  We 
also note that if $\cV = \{ 0\}$ then \eqref{Eqn2.2} yields a module version 
of the L\^{e}--Greuel formula.  We next see that \eqref{Eqn2.2} can also be 
viewed as computing the singular Milnor number of $F$ for a free divisor 
on a smooth subspace $\C^N \subset \C^{p +N}$.  This is the form that many 
terms on the RHS of \eqref{Eqn0.2} will take in the formulas we obtain.  
\begin{Proposition}
\label{ProptysingMil}
Let $\cV, 0 \subset \C^N, 0$ be an $H$--holonomic free divisor.
\begin{enumerate}
\item
\label{enit:ProptysingMil1}
Let $\cV^{\prime} = \cV \times \C^p, 0 \subset \C^{N+p}, 0$, and 
suppose $f_0 \co \C^n, 0 \to \C^{N+p}, 0$ is transverse to $\cV^{\prime}$ off 
$0$.  Then for $\pi$ denoting the projection $\C^{N+p} \to \C^N$,
$$ \mu_{\cV^{\prime}}(f_0) \,\, =  \,\, \mu_{\cV}(\pi \circ f_0) \, . $$
\item
\label{enit:ProptysingMil2}
Let $\cV^{\prime \prime}, 0 = \cV \times \{ 0\} \subset \C^{N+ p}, 
0$ be the image of $\cV, 0$ via the inclusion $\C^N, 0 \subset \C^{N+ p}, 0$ 
(so that $\cV^{\prime \prime}$ is a free divisor in a linear subspace of 
$\C^{N+ p}$).  Suppose $f_0\co \C^n, 0 \to \C^{N+p}, 0$ is transverse to 
$\cV^{\prime\prime}$ off $0$ and for $\pi^{\prime}$ denoting the 
projection $\C^{N+p} \to \C^p$, $\varphi = \pi^{\prime} \circ f_0\co \C^n, 0 
\to \C^{p}, 0$ is an ICIS.   Then 
$$ \mu_{\cV^{\prime\prime}}(f_0) \,\, =  \,\, \mu_{\varphi, \cV}(\pi \circ 
f_0) \, . $$ 
\end{enumerate}
\end{Proposition}
\begin{proof}
For \eqref{enit:ProptysingMil1},
we first note that $\cV^{\prime}$ is also $H$--holonomic.  If $\{ 
S_i\}$ are the strata of the canonical Whitney stratification of $\cV$, 
then $\{ S_i \times \C^p\}$ are the strata for $\cV^{\prime} = \cV \times 
\C^p$.  Also, if $\dlog (\cV)$ has the set of free generators $\eta_1, \dots 
\eta_{N-1}$ and we use coordinates $(w_1, \dots , w_p)$ for $\C^p$, then 
we can trivially extend the $\eta_i$ to $\C^{N+p}$ and adjoin $\left\{\pd{ 
}{w_1}, \dots \pd{ }{w_p}\right\}$ to obtain a set of free generators for 
$\dlog (\cV^{\prime})$.  Thus, $\cV^{\prime}$ is also $H$--holonomic.  

By a calculation similar to that for $\cK_{V, e}$ in \cite{D3}, it follows 
that for any germ $f_0\co \C^n, 0 \to \C^{N+p}$, with $\pi \co \C^{N+p} \to 
\C^{N}$ the projection,  $\cV$ defined by $H$, and $\cV^{\prime}$ defined 
by $H^{\prime} = H \circ \pi$, we have an isomorphism of normal spaces 
$$  \cK_{H^{\prime}, e}\cdot  f_0 \,\, \simeq \,\,\cK_{H, e} \cdot \pi\circ 
f_0  \, .  $$
Then, by \fullref{ThmAFD} we have \eqref{enit:ProptysingMil1}.
 
For \eqref{enit:ProptysingMil2},
we observe that if we choose a stabilization $f^{\prime}_t$ of $\pi 
\circ f_0$ so that $0 \notin f^{\prime -1}_t(\cV)$ for $t \neq 0$, then 
$F_t = (\varphi, f^{\prime}_t)$  is a stabilization of $f_0$ for 
$\cV^{\prime\prime}$.  Thus, the singular Milnor fiber of $\pi \circ f_0 | 
X$ for $\cV$, where $X = \varphi^{-1}(0)$, is also the singular Milnor fiber 
of $f_0$ for $\cV^{\prime\prime}$.  This yields \eqref{enit:ProptysingMil2}.
\end{proof}
\begin{Remark}
\label{Rem2.2}
In the formula \eqref{Eqn0.1}, if $\cW_i \subset \C^N$ has codimension 
$k$, then if $n < k$, the corresponding singular Milnor fiber of  $f_0
\co \C^n, 
0 \to \C^N, 0$ for $\cW_i$ will be empty and hence have Euler 
characteristic $0$.  Likewise, if $n-p <k$ then for $X, 0 \subset \C^n,0$ an 
ICIS defined by $\varphi \co \C^n, 0 \to \C^p, 0$, the singular Milnor fiber of  
$f_0 \co X, 0 \to \C^N, 0$ will be empty and hence have Euler characteristic 
$0$.   Thus, to make all of the formulas correct, we adopt the following 
convention:

{\bf Convention}\stdspace
{\it
If $n < k = \codim(\cW_i)$, then  $\mu_{\cW_i}(f_0) 
\overset{\rmdef}{=} (-1)^{n-k+1}$. Likewise if $n-p < k = \codim(\cW_i)$, then 
$\mu_{\varphi, \cW_i}(f_0) \overset{\rmdef}{=} (-1)^{n-p-k+1}$.
}
\end{Remark}

\begin{Remark}
The terms on the LHS of \eqref{Eqn2.2} can be viewed as computing the 
``relative singular Milnor number'', which is given by $\rank( 
H^{n-p-1}(X_t \cap B_{\gevar}, \cV_t \cap X_t \cap B_{\gevar}; \Z))$, 
where $X_t$ is the Milnor fiber of $\varphi$ and $\cV_t = f_t^{-1}(\cV)$.  
This follows because $\cV_t \cap X_t \cap B_{\gevar} \simeq \cV_t \cap 
X \cap B_{\gevar}$.  Since each fiber is homotopy equivalent to a bouquet 
of spheres, the exact sequence for a pair yields the sum on the LHS of 
\eqref{Eqn2.2}.
\end{Remark}
\subsection{Singular Vanishing Euler Characteristic} 
In the case that $\cV$ is not a complete intersection, we can still 
introduce a version of the vanishing Euler characteristic for the singular 
Milnor fiber (which may no longer be homotopy equivalent to a bouquet of 
spheres).  We suppose again that $f_0 \co \C^n, 0 \to M, 0$ is transverse to 
$\cV$ off $0$, and consider a stabilization $f_t\co B_{\gevar} \to M$ of 
$f_0$.  We let the \emph{singular vanishing Euler characteristic} be defined 
by
$$  \tilde \chi_{\cV}(f_0) \, \, \overset{\rmdef}{=} \, \, \tilde \chi\! \left(f_t^{-
1}(\cV)\right) \, \, = \, \, \chi \!\left(f_t^{-1}(\cV)\right) - 1 \, . $$
As earlier, $\tilde \chi_{\cV}(f_0)$ is independent of stabilization.
 
Similarly, if $X, 0$ is an ICIS defined by $\varphi \co \C^n, 0 \to \C^p$ and 
$f_0 \co X, 0 \to \C^N$ is transverse to $\cV$ off $0$, we define 
$$  \tilde \chi_{\varphi, \cV}(f_0) \, \, \overset{\rmdef}{=} \, \, \tilde \chi 
\!\left(f_t^{-1}(\cV \cap X)\right) \, \, = \, \, \chi \!\left(f_t^{-1}(\cV \cap X)\right) - 1\, .  $$
This can be viewed as the singular vanishing Euler characteristic for the 
mapping $F_0 = (\varphi, f_0) \co \C^n, 0 \to \C^p \times \C^N, 0$ since if 
$f_t | X \co X \cap B_{\gevar} \to \C^N$ is transverse to $\cV$, then $F_t = 
(\varphi, f_t) \co B_{\gevar} \to \C^p \times \C^N$ is transverse to $\{ 0\} 
\times \cV$.  Thus, $\tilde \chi_{\varphi, \cV}(f_0) = \tilde \chi_{\{ 0\} 
\times \cV}(F_0)$.  

We will compute singular Milnor numbers for nonlinear sections of 
hypersurface and complete intersection singularities.  However, we will 
do so by using simple Euler characteristic arguments for the singular 
vanishing Euler characteristics combined with their calculation in terms 
of singular Milnor numbers.  These, in turn, can be calculated algebraically 
using \eqref{Eqn2.1} and \fullref{ThmAFDICIS}.  The simplest version 
is for the case of subvarieties $\cV, \cW \subset \C^N$.  
\begin{Lemma}
\label{Lem2.9}
Suppose $f_0 \co \C^n, 0 \to \C^N, 0$ is transverse to $\cV$, $\cW$ and $\cV 
\cap \cW$ off $0 \in \C^n$.  Then, 
\begin{equation}
\label{Eqn2.9}
 \tilde \chi_{\cW \cup\cV}(f_0) \quad = \quad \tilde \chi_{\cW }(f_0) \, + 
\,  \tilde \chi_{\cV}(f_0) \,- \, \tilde \chi_{\cW \cap \cV}(f_0) \, .
\end{equation}
In the case that $\cV$ and $\cW$ are both hypersurface singularities we 
obtain from \eqref{Eqn2.9}
\begin{equation}
\label{Eqn2.10}
\mu_{\cW}(f_0) \quad = \quad \mu_{\cW \cup \cV}(f_0) - \mu_{\cV}(f_0) 
\, + \,  
(-1)^{n-1} \tilde \chi_{\cW \cap \cV}(f_0) \, .
\end{equation}
If instead $X, 0$ is an ICIS defined by $\varphi \co \C^n, 0 \to \C^p, 0$ and 
$f_0 \co X, 0 \to \C^N, 0$ is transverse to $\cV$ and $\cW$ off $0$, then 
there are the analogs for \eqref{Eqn2.9} and \eqref{Eqn2.10}
\begin{equation}
\label{Eqn2.9a}
 \tilde \chi_{\varphi, \cW \cup\cV}(f_0) \quad = \quad  \tilde 
\chi_{\varphi, \cW }(f_0) \, + \,   \tilde \chi_{\varphi, \cV}(f_0) \,- \, 
\tilde \chi_{\varphi, \cW \cap \cV}(f_0) 
\end{equation}
and 
\begin{equation}
\label{Eqn2.10a}
\mu_{\varphi, \cW}(f_0) \quad = \quad \mu_{\varphi, \cW \cup \cV}(f_0) 
- \mu_{\varphi, \cV}(f_0) \, + \,  
(-1)^{n-p-1} \tilde \chi_{\varphi, \cW \cap \cV}(f_0) \, .
\end{equation}
\end{Lemma}
\begin{notation}
To simplify formulas, we will view singular Milnor 
numbers and singular vanishing Euler characteristics as numerical 
functions on the space of germs transverse to the appropriate set of 
subvarieties off $0$.  Hence, a formula such as \eqref{Eqn2.10} will be 
written with evaluation on $f_0$ understood so it will take the form 
\begin{equation}
\label{Eqn2.10b}
\mu_{\cW} \quad = \quad \mu_{\cW \cup \cV} - \mu_{\cV} \, + \,  
(-1)^{n-1} \tilde \chi_{\cW \cap \cV} \, .
\end{equation}
Also, we may apply \fullref{ProptysingMil} to obtain $\mu_{\pi^* 
\cE}(f_0) = \mu_{\cE}(\pi \circ f_0)$, so with this understanding, in all 
future formulas we will abbreviate $\mu_{\pi^* \cE}$ to just 
$\mu_{\cE}$.
\end{notation}
\begin{proof}[Proof of \fullref{Lem2.9}]
The  addition-deletion type argument for reduced Euler characteristics 
($\tilde \chi = \chi -1$) for subvarieties applied to the hypersurfaces 
$\cW$ and $\cV$ give \eqref{Eqn2.9}.  Then, for a hypersurface $\cW$, we 
have $\tilde \chi_{\cW }(f_0) = (-1)^{n-1}\mu_{\cW }(f_0)$.  Substituting 
for $\tilde \chi_{}$ for all of the hypersurfaces in \eqref{Eqn2.9} and 
rearranging yields \eqref{Eqn2.10}. 
 
The same Euler characteristic argument used in verifying \eqref{Eqn2.9} 
also applies instead to $\{ 0 \} \times \cY \subset \C^{p + N}$ for 
hypersurfaces $\cY$ and the map $F = (\varphi, f_0)$ yielding 
\eqref{Eqn2.9a}.  Substituting $\tilde \chi_{\varphi, \cW }(f_0) = (-1)^{n-
p-1}\mu_{\varphi, \cW }(f_0)$ for all of the hypersurfaces in 
\eqref{Eqn2.9a} yields after rearranging \eqref{Eqn2.10a}. 
\end{proof}
\subsection{Intersections of Multiple Hypersurfaces}
To compute $\tilde \chi_{\cV \cap \cW}$ we will use an inductive 
procedure which requires computing $\tilde \chi_{\cap_{i} \cW_i}$ for a 
collection of hypersurfaces $\cW_i$.  We will use the following formula 
for $k$ hypersurfaces $\cW_i$ :  
\begin{equation}
\label{Eqn2.13}
\tilde \chi_{\cap_{i} \cW_i} \,\, = \,\, \sum_{\bj} (-1)^{| \bj | + 1}\tilde 
\chi_{\cup_{\bj} \cW_{j_i}}
\end{equation}
 for nonempty $\bj = \{j_1, \dots , j_r\} \subset \{1, \dots , k\}$ with $| 
\bj | = r$ (for a formula involving $\chi$ see \cite[Lemma 8.1]{D2}, but an 
analogous addition-deletion argument works for $\tilde \chi$ using 
reduced homology).

Then, for mappings $f_0 \co \C^n, 0 \to \C^N, 0$, substituting $\tilde 
\chi_{\cup_{\bj} \cW_{j_i}} = (-1)^{n-1}\mu_{\cup_{\bj} \cW_{j_i}}$ we 
obtain
\begin{Proposition}
\label{Prop2.14}
For mappings $f_0 \co \C^n, 0 \to \C^N, 0$ and a collection of hypersurfaces 
$\cW_i, 0 \subset \C^N, 0$, $i = 1, \dots , k$, with $\cap_{i} \cW_i$ not 
necessarily a complete intersection,
\begin{equation}
\label{Eqn2.14}
\tilde \chi_{\cap_{i} \cW_i} \,\, = \,\, (-1)^{n-k}\left(\sum_{\bj} (-1)^{| 
\bj | + k} \mu_{\cup_{\bj} \cW_{j_i}} \right) \, .
\end{equation}
\end{Proposition}
\begin{Remark}
\label{Rem2.14}
In the case that $\cap_{i} \cW_i$ is a complete intersection, this formula 
reduces to Theorem 2 of \cite[\S 8]{D2}.
\end{Remark}
\section{Exceptional Orbit Varieties as Free Divisors}  
\label{S:sec3} 
We recall the results from part I \cite{DP1} which allow us to embed the 
varieties of singular matrices in a geometric configuration of divisors 
which form free divisors. 
 
We use the notation from part I and let $M_{m, p}$ denote the space of $m 
\times p$ complex matrices, and $\Sym_m$, respectively $\Sk_m$, the 
subspaces of $M_{m, m}$ of symmetric, respectively skew-symmetric, 
complex matrices.  Next, we let $B_m$ denote the Borel subgroup of 
$\GL_m(\C)$ consisting of lower triangular matrices and the group
$$   C_m = \begin{pmatrix} 1 &0 \\ 0 & B_{m-1}^T \end{pmatrix}  $$
where $B_{m-1}^T$ denote the group of upper triangular matrices of 
$\GL_{m-1}(\C)$.  Then, the (modified) \emph{Cholesky-type 
representations} are given in \fullref{table2.0}, which is Table 1 of 
\cite{DP1}.  These representations give rise to \emph{exceptional orbit 
varieties} which are the union of the positive codimension orbits of the 
representations.  We denote these by: $\cE^{sy}_m$ (for $\Sym_m$); 
$\cE_m$ (for $M_{m, m}$); $\cE_{m-1, m}$ (for $M_{m-1, m}$); and 
$\cE^{sk}_m$ (for $\Sk_m$).  Then, by \cite{GMNS} for the symmetric case 
and for all cases by Theorems $6.2$, $7.1$, and $8.1$ in \cite{DP1}, the 
first three families are linear free divisors, and the last $\cE^{sk}_m$ are 
free divisors.  These are families of representations which, via natural 
inclusions of groups and spaces, together form towers of representations.  
Furthermore, the exceptional orbit varieties contain as components the 
corresponding ``generalized determinant varieties'', which we 
denote by:  $\cD^{sy}_m$, $\cD_{m}$, $\cD_{m-1, m}$, and $\cD^{sk}_m$ 
respectively.  The defining equations for the corresponding exceptional 
orbit varieties and generalized determinant varieties are given in
\fullref{table3.0}.  
Because of the tower structure for the representations we have the 
inductive representation for the $m$--th exceptional orbit variety 
$\cE_m$ and generalized determinant variety $\cD_m$
\begin{equation}
\label{Eqn3.1}
  \cE_m \,\, = \,\, \cD_m \cup \pi^*\cE_{m-1}\, , 
\end{equation}
where $\pi$ denotes a projection from the $m$--th representation $V_m$, 
$\pi \co V_m \to  V_{m-1}$.  
Then, by \eqref{Eqn3.1}, in each case $\cD_m$ has a free completion to 
$\cE_m$ by $\pi^*\cE_{m-1}$.  
\begin{table}
\begin{center}
{\it  (Modified) Cholesky-Type Representations Yielding Free 
Divisors}
\begin{tabular}{lccl}
\hline
(Modified) Cholesky-  & Matrix  & Solvable  & Representation \\
\quad type factorization  & space    &  group    &     \\
\hline
Symmetric matrices & $\Sym_m$ & $B_m$  &  $B\cdot A = B\, A\, B^T$\\
General $m \times m$ & $M_{m, m}$  &  $B_m \times C_m$   &   $(B, 
C)\cdot A = B\, A\, C^{-1}$   \\
General $(m-1) \times m$  & $M_{m-1, m}$ & $B_{m-1} \times C_m$   &   
$(B, C)\cdot A = B\, A\, C^{-1}$     \\
\hline \hline
Nonlinear representation         &     Matrix   & Solvable       &    Representation \\
&   space    &  Lie algebra   &        \\
\hline 
Skew-symmetric matrices & $\Sk_m$ & $\tilde D_m$  & $\Diff(\cE^{sk}_m, 
0)$  \\
\end{tabular}
\caption{\label{table2.0} Solvable group and solvable Lie algebra
block representations for (modified) Cholesky-type factorizations, 
yielding the free divisors in \fullref{table3.0}.}
\end{center}
\end{table}
\begin{Remark}
\label{Rem3.4}
For $\Sk_m$, in place of a solvable group, we have an infinite dimensional 
solvable Lie algebra $\tilde D_m$ which is an extension of the Lie algebra 
of the solvable Lie group 
$$ G_m = \begin{pmatrix}  T_2 & 0_{2, m-2} \\ 0_{m-2, 2}  & B_{m-2}   
\end{pmatrix}  $$
where $T_2$ is the group of $2 \times 2$ diagonal matrices. 
This extension is by a set of Pfaffian vector fields $\eta_k$ for $2 \leq k 
\leq m-2$, see \cite[\S 8]{DP1} and \cite[Chap. 5]{P}.  The resulting 
infinite dimensional Lie group $\Diff(\cE^{sk}_m, 0)$ is the group of germs 
of diffeomorphisms preserving $\cE^{sk}_m$.
\end{Remark} 
\begin{table}
\begin{center}
\begin{tabular}{lclc} 
$\cE$      &   Defining Equation for $\cE$   & $\cD$  & Defining Equation 
for $\cD$ \\
\hline
$\cE^{sy}_m$ &  $\displaystyle\prod_{k =1}^{m}
\det\!\left(A^{(k)}\right)$   &  $\cD^{sy}_m$  &  $\det(A)$  \\
$\cE_m$ &  $\displaystyle\prod_{k =1}^{m} \det\!\left(A^{(k)}\right)\cdot \prod_{k =1}^{m-
1}\det\!\left(\hat A^{(k)}\right)$  &  $\cD_m$   &  $\det\!\left(\hat A^{(m-1)}\right)\cdot \det(A)$  
\\
$\cE_{m-1,m}\!\!$  &  $\displaystyle\prod_{k =1}^{m-1} \det\!\left(A^{(k)}\right)\cdot \prod_{k 
=1}^{m-1}\det\!\left(\hat A^{(k)}\right)$   & $\cD_{m-1, m}\!\!$   &  $\det\!\left(A^{(m-1)}\right)\cdot 
\det\!\left(\hat A^{(m-1)}\right)$     \\
$\cE^{sk}_m$ & $\displaystyle\prod_{k=1}^{m-2} \det\!\!\left({\hat{\hat A}}^{(k)}\right) 
\cdot
\prod_{k=2}^m \Pf_{\{\epsilon(k),\ldots,k\}}(A)$  & $\cD^{sk}_m$  & 
$\Pf_{\{\epsilon(m),\ldots,m\}}(A)\cdot \det\!\!\left({\hat{\hat A}}^{(m-2)}\right)$  \\
\end{tabular}
\caption{\label{table3.0} Defining equations for the exceptional orbit 
varieties $\cE$ and determinantal varieties $\cD$ for the solvable group 
and solvable Lie algebra block representations
in \fullref{table2.0}.
If $A  = (a_{i j})$ denotes 
a general matrix, then $\hat A$ denotes the matrix obtained by deleting 
the first column of $A$ and $\hat{\hat A}$, that obtained by deleting the 
first two columns of $A$.  Then, $A^{(k)}$ denotes the $k \times k$ upper 
left-hand submatrix of a matrix $A$.  Also, 
$\Pf_{\{\epsilon(k),\ldots,k\}}(A)$ denotes the Pfaffian of the 
skew-symmetric submatrix of $A$ consisting of the consecutive rows and 
columns $\epsilon(k),\ldots,k$, where $\epsilon(k) = 1, 2$ with 
$\epsilon(k) \equiv k+1 \mod 2$.}
\end{center}
\end{table}
\begin{Remark}
\label{Rem3.5}
We may interleave the towers of general matrices so $M_{m-1, m-1} 
\subset M_{m-1, m} \subset M_{m, m}$.  Then, the successive generalized 
determinantal varieties are defined by $\det\!\left(\hat A^{(m-1)}\right)$ and then 
$\det(A)$.  
\end{Remark}
\subsection{Free Divisors arising from Restrictions of Block 
Representations}
In addition to the free divisors arising from the representations in 
\fullref{table2.0}, we shall also use certain auxiliary free divisors arising 
from the restriction of representations.  These are given in \S 9 of 
\cite{DP1}.  

For $\Sym_3$ we use coordinates given by
$$ A \,\, = \,\, \begin{pmatrix}
a & b & c \\
b  & d & e \\
c & e & f 
\end{pmatrix}\, .$$
We define $\cQ_f  = \det(A_f)$ and $\cQ_a = \det(A_a)$
where $A_f$ and $A_a$ are obtained from $A$  by setting $f = 0$, 
respectively, $a = 0$.  Interchanging the first and third coordinates in 
$\C^3$ will interchange $\cQ_f$ and $\cQ_a$ so any result for $\cQ_f$ 
will have an analogous result for  $\cQ_a$.  We let $V_a$ denote the 
subspace where $a = 0$ and $V_f$, where $f = 0$.  Then, we can summarize 
the appropriate results from Propositions 9.1 and 9.5 of \cite{DP1}.  
\begin{Proposition} 
\label{Prop3.4}
The subvarieties of $V_a$ defined by $b\cdot d\cdot \cQ_a = 0$ and of 
$V_f$ defined by $\left(ad - b^2\right)\cdot \cQ_f = 0$ are linear free divisors.  
\end{Proposition}
Hence, by \fullref{Prop3.4}, $V(\cQ_a)$ has a free completion 
using the free divisor $V(bd)$, and we may complete $V(\cQ_f)$ to a free 
divisor using $\cD_2^{sy} = V(ad - b^2)$.  Although $\cD_2^{sy}$ is not a 
free divisor, it has a free completion $\cE_2^{sy}$.  
\subsection{A Quiver Linear Free Divisor}
A third special case of linear free divisors needed for our calculations 
occurs for the special case of $2 \times 3$ matrices.  In \cite{BM}, 
Buchweitz and Mond proved that quivers of finite type give rise to free 
divisors.  The quiver consisting of $3$ arrows from vertices (representing 
$\C$) to a central vertex (representing $\C^2$) corresponds to the 
representation of $\left(\GL_2(\C) \times (\C^*)^3\right)\!/\C^*$ on $M_{2, 3}$.  If we 
use coordinates on $M_{2, 3}$ given by 
$\begin{pmatrix}
a & b & c \\
d  & e & f 
\end{pmatrix}$, 
then the corresponding free divisor is defined by 
$(a e - b d) (a f - c d) (b f - c e) = 0$.  
\subsection{Linear Free Divisors which are $H$--holonomic}
\fullref{ThmAFD} allows us to compute $\mu_{\cV}(f_0)$ provided
$\cV$ is an $H$--holonomic free divisor.  In this section we give two 
results establishing that free
divisors are $H$--holonomic; one applies to towers of linear free
divisors, and the other, to arbitrary low-dimensional linear free divisors.
\subsubsection*{$H$--holonomic free divisors which appear in towers}
Let $\cE$ be a free divisor arising as the exceptional orbit variety of
a representation $G \to \GL(W)$, which itself is one step of a tower
of representations as defined in part I (\cite{DP1}).  
For example, $\cE$ could be any of the hypersurfaces in the following
Theorem, which is proven in detail in \S6.3 of \cite{P} using
the technique we will describe.
\begin{Thm}[Theorem 6.2.2 in \cite{P}]
\label{thm:towerhholo}
The linear free divisors $\cE_m^{sy}$, $\cE_m$, and $\cE_{m-1,m}$ listed 
in \fullref{table3.0} are $H$--holonomic.
\end{Thm}
\begin{proof}[Outline of Proof]
We outline what is a fairly lengthy argument which is proven in detail in 
\S 6.3 of \cite{P}.  Readers are encouraged to refer there for the full 
details.

First, it is proven that there are only a finite number of orbits of $G$ in 
$W$ by classifying them, giving normal forms for representatives of each 
orbit.  The tower structure makes this step significantly easier, because 
the classification at a lower level of the tower can be combined with the 
inclusion of the  group action and vector spaces to put an arbitrary $w\in 
W$ into a ``partial normal form'' $g_1\cdot w$ (for example, a 
certain submatrix of $g_1\cdot w$ contains only zeros and ones in a 
certain pattern).  Then, another element of $G$ is applied to put $g_1\cdot 
w$ into a normal form.  As the resulting list of normal forms is finite, 
there are a finite number of $G$--orbits in $W$ (and thus in the 
exceptional orbit variety $\cE$), and so $\cE$ is holonomic.

Second, we let $G_H\subset G$ be the connected codimension $1$ Lie 
subgroup whose Lie algebra of vector fields generates $\dlog(H)$.
To show $\cE$ is $H$--holonomic, it is sufficient to prove that $G_H$
acts transitively on all non-open $G$--orbits (or, the $G$--orbits in
$\cE$ are the $G_H$--orbits in $\cE$).
Thus we consider each normal form $n$ (representing a non-open orbit) 
with an arbitrary $g\in G$, and show that there exists an $h\in G$ in the 
isotropy subgroup of $n$ with $hg\in G_H$.  
Thus, if $n=g\cdot v$ then $n=hg\cdot v$ with $hg\in G_H$.
It follows that $G\cdot n=G_H\cdot n$.
\end{proof}
\subsubsection*{$H$--holonomic free divisors in small dimensions}
Since we use other linear free divisors described above, we
also provide the following sufficient condition for a hypersurface
to be $H$--holonomic.  In low dimensions, the criterion 
can be checked by a computer using a computer algebra system such as 
Macaulay2 or Singular.  

Let $\cV,0\subset \C^n, 0$ be a reduced hypersurface with good defining 
equation $H$.  Let $M$ be an $\cO_{\C^n, 0}$--module of vector fields on 
$\C^n, 0$.  We let for $z \in \C^n$,
$$\langle M \rangle_{(z)}\,\, =\,\,  \{\eta(z)\, | \, \eta\in M\}  $$ 
be the linear subspace of $T_z \C^n$.  The logarithmic and 
$H$--logarithmic tangent spaces are defined to be 
$$ T_{\rmlog}\cV_{z} \,\, = \,\, \langle \dlog(V)\rangle_{(z)}  \quad \makebox{ and } 
\quad T_{\rmlog}H_{z} \,\, = \,\, \langle \dlog(H)\rangle_{(z)} \, .$$
For $0\leq k\leq n$, define the varieties
$D_k=\{z\in \cV \, | \, \dim(T_{\rmlog}\cV_{z}) \leq k \}$ and
$H_k=\{z\in \cV\, | \, \dim(T_{\rmlog}H_{z}) \leq k \}$.
\begin{Proposition}
\label{prop:hholonomic} 
With the preceding notation, if, for all $0 \leq k <n$,
\begin{enumerate}
\item
\label{enit:hholonomici}
all irreducible components of $(D_k,0)$ have
dimension $\leq k$ at $0$, and
\item
\label{enit:hholonomicii}
$(D_k,0)=(H_k,0)$ as germs,
\end{enumerate}
then $(\cV,0)$ is $H$--holonomic. 
\end{Proposition}
\begin{proof}
For $z\in \cV$, let $S_z$ denote the stratum of the canonical Whitney 
stratification of $\cV$ containing $z$.  Then, $\cV$ is holonomic if 
and only if $T_{\rmlog}\cV_{z} = T_zS_z$ for all $z \in \cV$, and it is 
$H$--holonomic if and only if $T_{\rmlog}H_{z} = T_zS_z$ for all $z \in 
\cV$.

First, we observe that the conditions imply $\cV$ is holonomic for if not, 
then there is a stratum $S$ of highest dimension, say $k$, on which it 
fails.  Then, there is a Zariski open set $U$ of $S$ consisting of those $z 
\in S$ with $T_{\rmlog}\cV_{z} \subsetneq T_zS_z$.  Then, $U \subset D_{k-
1}$, and $\dim D_{k-1} \geq k$, contradicting \eqref{enit:hholonomici}.
A similar argument using 
$T_{\rmlog}H_{z}$ shows if $\cV$ is not $H$--holonomic, then $\dim D_{k-1} 
\geq k$, contradicting \eqref{enit:hholonomicii} given that
\eqref{enit:hholonomici} holds.
\end{proof}

Computer algebra systems such as Macaulay2 and Singular have built-in 
functions to perform each of the steps necessary to use
\fullref{prop:hholonomic} to show that a hypersurface is 
$H$--holonomic, including:  finding generators of $\dlog(V)$ and
$\dlog(H)$ (as certain syzygies), determining the ideals defining each 
$D_k$ and $H_k$, computing the radicals and primary decompositions of 
these ideals, computing the dimensions of the irreducible components of 
$D_k$, and testing pairs of ideals for equality.  
\begin{Remark} 
\label{Rem3.6}
In particular, the linear free divisors in \fullref{Prop3.4} and the 
quiver linear free divisor in $M_{2, 3}$ are $H$--holonomic.
\end{Remark}  
When we assert that a hypersurface is an $H$--holonomic free divisor and
give no reference, it will be understood that we have used an 
implementation (\cite{P2})  of this approach in Macaulay2 
(\cite{M2}) to check Saito's Criterion and the conditions of 
\fullref{prop:hholonomic}.  
\section{A Metatheorem and Generic Reduction}  
\label{S:sec4} 
In this section we introduce two ideas which both extend and simplify the 
formulas for singular Milnor numbers which we will obtain.
\subsection{Metatheorem}
The results on matrix singularities for $f_0 \co \C^n, 0 \to M, 0$ can be 
extended to the case of matrix singularities on an ICIS X.  In fact given a 
formula \eqref{Eqn0.2} for $\mu_{\cV}$, the following metatheorem 
asserts that there is a corresponding formula for the singular Milnor 
number of $f_0 | X, 0 \to M, 0$. 
\begin{metatheorem}
\label{MetaThm1}
If $X$ is an ICIS defined by $\varphi \co \C^n, 0 \to \C^p, 0$, and the formula 
\eqref{Eqn0.2} for $\mu_{\cV}$ is obtained by the inductive procedure, 
then the same procedure also yields the formula (with the same 
coefficients $a_i$)
\begin{equation}
\label{Eqn0.3a}
 \mu_{\varphi, \cV} \,\,  = \,\,  \sum_{i} \, a_i \mu_{\varphi, \cW_i}
\end{equation}
where $\mu_{\varphi, \cV}(f_0)$, respectively $\mu_{\varphi, 
\cW_i}(f_0)$, are the singular Milnor numbers for $f_0 | X$ as nonlinear 
sections of $\cV$, resp. $\cW_i$, and can be computed as lengths of 
determinantal modules.

Likewise, if instead we have a formula for the vanishing Euler 
characteristic $\tilde \chi_{\cV}$ having the same form as in 
\eqref{Eqn0.2} 
\begin{equation}
\label{Eqn0.2b}
 \tilde \chi_{\cV} \,\,  = \,\,  \sum_{i} \, b_i \mu_{\cW_i} 
\end{equation}
and obtained by the inductive process, then there is an analogous formula
\begin{equation}
\label{Eqn0.3b}
 \tilde \chi_{\varphi, \cV} \,\,  = \,\,  (-1)^p\left( \sum_{i} \, b_i 
\mu_{\varphi, \cW_i} \right) \, . 
\end{equation}
\end{metatheorem}
\begin{proof}
This result follows because at each inductive step, the decomposition into 
the associated varieties will be the same.  Then, in place of using the 
formulas in \fullref{Lem2.9} and \fullref{ThmAFD} for germs 
$f_0$ on $\C^n$, we use the versions of \fullref{Lem2.9} for $f_0 | X$ 
on an ICIS $X$ and \fullref{ThmAFDICIS}.  Also, for a variety in $M$ 
defined by $(g_1, \dots , g_r)$, in place of $\mu_{g_1, \dots , g_r}(f_0)$ 
we use 
$\mu_{(g_1, \dots , g_r) \circ \pi}((\varphi, f_0))$, with $\pi \co \C^{r + p} 
\to \C^r$ denoting the projection.  This we denote by $\mu_{\varphi, g_1, 
\dots , g_r}(f_0)$.  This can be seen by observing that in terms of singular 
vanishing Euler characteristics, we repeatedly use \eqref{Eqn2.9} from 
\fullref{Lem2.9}.  However, for $f_0 | X$ we repeatedly use instead 
\eqref{Eqn2.9a}.  Thus, the formulas in terms of singular vanishing Euler 
characteristics will have the same form.  However, in writing  the 
formulas in terms of singular Milnor numbers, $\tilde \chi_{\cW_i} = (-
1)^{n-k}\mu_{\cW_i}$ where $k$ is the codimension of $\cW_i$; while   
$\tilde \chi_{\varphi, \cW_i} = (-1)^{n-p -k}\mu_{\varphi,\cW_i}$.  Since 
the extra factor of $(-1)^p$ will occur for every term on each side, it will 
cancel yielding \eqref{Eqn0.3a}.  However, for  $\tilde \chi_{\varphi, \cV}$ 
versus $\tilde \chi_{\cV}$, there is an extra factor of $(-1)^p$ for each 
term on the RHS, resulting in the desired formula \eqref{Eqn0.3b}. 
\end{proof}
\subsection{Generic Reduction}
Given a matrix singularity defined by $f_0$, we may apply an element $g$ 
of the group $G$ which acts on the space of matrices $M$ to obtain $f_1 = 
g\cdot f_0$ which is $\cK_M$--equivalent to $f_0$ and has the same singular 
Milnor number.  By \fullref{Rem1a.1} we can apply $g$ so that $f_1$ 
is transverse to the associated varieties, allowing us to compute 
$\mu_{\cD}(f_0)$ using formulas of the form \eqref{Eqn0.2}.  However, we 
can do more and this leads to the idea of \emph{generic reduction}.

We can simplify the form which the formulas take if we can choose $f_1$ 
so as many of the terms in \eqref{Eqn0.2} vanish.  We can achieve this by 
considering $df_0(0)$ and the effect of applying $g$ to it to obtain 
$df_1(0)$.  

Given $\cW_i, 0$, we choose $M_i \subset M$ as the linear subspace of 
minimal dimension containing $\cW_i$.  We also represent $\cW_i, 0$ as 
the pullback of a divisor by the projection $ \pi_i \co M_i \to \C^{m_i}$, for   
minimal $m_i$.  Then, the \emph{defining dimension} of $\cW_i$ is $\codim 
\, M_i + m_i$, and the \emph{defining codimension} of $\cW_i$ is $\dim M_i 
- m_i$.  We then let $\gl_{\ell}$ denote the sum of the  terms in 
\eqref{Eqn0.2} for the $\cW_i$ of defining codimension $\ell$.  Then, by 
generic reduction we mean that an element $g$ of $G$ is applied so that 
$df_1(0)$ projects submersively onto each $M/\ker(\pi_i)$ for those  
$\cW_i$ of defining codimension $\geq \codim (\Im(df_1(0)))$.  Then, all of 
the terms $\gl_{\ell}(f_1)$ will be $0$ for $\ell \geq \codim 
(\Im(df_1(0)))$.

In certain cases, the classification of linear matrix singularities may 
prevent us from obtaining an $f_1$ with the full generic reduction; 
however, we will still apply $g$ to obtain as many terms vanishing as 
possible.  The results obtained in the later sections will indicate how 
generic reduction simplifies the formulas.  In \S \ref{S:sec9} we deduce 
specific consequences of generic reduction for all of the matrix types for 
generic corank $1$ matrix mappings and for the computations for 
Cohen--Macaulay singularities.
\section{Symmetric Matrix Singularities}  
\label{S:symmatr} 
By the results of \cite{DP1} summarized in \S \ref{S:sec3}, the 
exceptional orbit variety $\cE_m^{sy}$ of the representation of $B_m$ on 
$\Sym_m$ is a linear free divisor and the determinantal variety 
$\cD_m^{sy}$ has a free completion given by 
\begin{equation}
\label{Eqn5.1}
     \cE_m^{sy} \,\, =   \pi^* \cE_{m-1}^{sy} \cup \cD_m^{sy}  
\end{equation}
for the projection $\pi \co \Sym_m \to \Sym_{m-1}$.

Furthermore, by \fullref{thm:towerhholo}, $\cE_m^{sy}$ is 
$H$--holonomic; hence by \fullref{ThmAFD}, for a nonlinear section 
$f_0 \co \C^n, 0 \to \Sym_m$, transverse to $\cE_m^{sy}$ off $0$, the 
singular Milnor number $\mu_{\cE_m^{sy}}$ is the length of the 
determinantal module 
$$   N \cK_{H, e}\,(f_0) \,\, \simeq  \,\, N \cK_{\tilde B_m, e}\,(f_0)$$  
where $\tilde B_m$ is the subgroup of $B_m$ which preserves the 
defining equation $H$ of $\cE_m^{sy}$.  The corresponding Lie algebra of 
representation vector fields is $\dlog(H)$.

Hence, by \fullref{Lem2.9} and \eqref{Eqn5.1}, we have quite generally 
\begin{equation}
\label{Eqn5.2}
  \mu_{\cD_m^{sy}} \,\, = \,\,  \mu_{\cE_m^{sy}} \, - \, 
\mu_{\cE_{m-1}^{sy}} \, + \, (-1)^{n-1}\, \tilde \chi_{\pi^* \cE_{m-1}^{sy} 
\cap \cD_m^{sy}}  \, .
\end{equation}
Thus, we are reduced to inductively computing 
$\tilde \chi_{\pi^* \cE_{m-1}^{sy} \cap \cD_m^{sy}}$.
We note that the simplest case of $\cD_1^{sy} = \{ 0\} \subset \Sym_1 
\simeq \C$  just yields isolated hypersurface singularities and 
$\mu_{\cD_1^{sy}} = \mu$ when applied to $f_0 \co \C^n, 0 \to \Sym_1, 0 
\simeq \C, 0$.  We have already carried out the calculation for $2 \times 
2$ symmetric matrices in \S \ref{S:sec1a} which leads to the following 
theorem.
\begin{Thm}
\label{Thmsym2}  
For the space of germs transverse to the associated varieties for 
$\cE_2^{sy}$ off $0$,  
\begin{equation}
\label{Eqn5.3}
  \mu_{\cD_2^{sy}}   \,\, = \,\,   \mu_{\cE_2^{sy}} \, -  \, (\mu_a \,  + \, 
\mu_{a, b}) 
\end{equation}
where $\mu_{\cE_2^{sy}} = \cK_{\tilde B_2, e}$--$\codim$ and $\mu_a   + 
\mu_{a, b}$ is the length of a determinantal module by the L\^{e}--Greuel 
formula (\fullref{ThmLeGr}).

By \fullref{MetaThm1} there is an analog of \eqref{Eqn5.3} 
for the Milnor number $\mu_{\varphi, \cD_2^{sy}}$ on the ICIS 
$X = \varphi^{-1}(0)$ defined by $\varphi \co \C^n, 0 \to \C^p, 0$.
\end{Thm}
\begin{proof}
We have already obtained \eqref{Eqn5.3}, and the metaversion follows from 
the Metatheorem.
\end{proof}
We observe that for germs $f_0 \co \C^2, 0 \to \Sym_2, 0$ transverse to 
$\cD_2^{sy}$ off $0$, $\det \circ f_0$ defines an isolated hypersurface 
singularity and the Milnor number $\mu(\det \circ f_0) = \dim \cO_{\C^2, 
0}/ \Jac(\det \circ f_0)$.  The Milnor fiber of $\det \circ f_0$ equals the 
singular Milnor fiber of $f_0$, and hence the Milnor number and singular 
Milnor number agree.  For $n > 3$ and $f_0 \co \C^n, 0 \to \Sym_2, 0$ 
(transverse to $\cD_2^{sy}$ off $0$), $\det \circ f_0$  no longer has an 
isolated singularity; however, the singular Milnor number is still defined.  

We consider the case where $f_0$ has rank $\geq 1$.  We may apply a 
matrix transformation on $\Sym_2$ so that $df_0(0)$ has nonzero 
upper-left entry.  Furthermore, we may assume that under the 
transformation, $f_0$ is transverse off zero to the line $a = b = 0$, so the 
composition of $f_0$ with projection onto the $(a, b)$--subspace has an 
isolated singularity at $0$.  Thus, after applying the transformation, we 
may apply a change of coordinates in $\C^n, 0$ so that for $y = (y_1, \dots, 
y_{n-1})$, $f_0$ has the form 
\begin{equation}
\label{Eqn5.3a}
 f_0(x, y) \,\, = \,\,  \begin{pmatrix}
x & g(x, y)  \\
 g(x, y)   &  h(x, y)  
\end{pmatrix}  \, . 
\end{equation}
In the case that $g$ is weighted homogeneous we can collapse 
\eqref{Eqn5.3} to yield a Jacobian-type formula for the singular Milnor 
number.  We let $g$ be weighted homogeneous of weighted degree $\ell$ 
for the weights $\wt (x, y_1, \dots , y_{n-1}) = (a_0, a_1, \dots , a_{n-
1})$ and Euler vector field  $e = a_0 x \pd{ }{x} + \sum a_i y_i \pd{ }{y_i}$.
\begin{Corollary}[Jacobian Formula]
\label{CorJacfor}
If $n \geq 2$ and $f_0 \co \C^n, 0 \to \Sym_2, 0$ has the form \eqref{Eqn5.3a} 
with $g$ weighted homogeneous (and is transverse to the associated 
varieties off $0$), then  
\begin{equation}
\label{Eqn5.4}
 \mu_{\cD_2^{sy}}(f_0) \,\, = \,\, \dim_{\C}\left( \cO_{\C^n, 
0}/(\widetilde{\Jac}(\det \circ f_0) + \Jac(f_0))\right) 
\end{equation}
where $\Jac(f_0)$ is the ideal generated by the $3 \times 3$ minors of 
$df_0$ and $\widetilde{\Jac}(\det \circ f_0)$ is a modified Jacobian ideal 
where $\pd{(\det \circ f_0)}{x}$ is replaced by $(2\ell +a_0)\pd{(\det 
\circ f_0)}{x} + \gd(h)$ for $\gd(h) = (2\ell -a_0)h - e(h)$.  If $\det \circ 
f_0$ is weighted homogeneous (for the same weights as $g$), then 
($\gd(h) = 0$ and)  $\widetilde{\Jac}(\det \circ f_0) = \Jac(\det \circ f_0)$.  
\end{Corollary}
\begin{Remark}
\label{Rem5.4}
In the Corollary, if $n = 2$ then there are no $3 \times 3$ minors, so the 
formula reduces to $\dim_{\C}( \cO_{\C^n, 0}/(\widetilde{\Jac}(\det \circ 
f_0))$.  If $\det \circ f_0$ is weighted homogeneous then this formula 
becomes Milnor's formula.  However, in general it differs from 
Milnor's formula by the addition of the term $\gd(h)$ to 
$(2\ell + a_0)\pd{(\det \circ f_0)}{x}$, although the dimension does not 
change.

In fact, since we are only computing dimensions, we suspect that the 
formula should be correct with $\Jac(\det \circ f_0)$ in place of 
$\widetilde{\Jac}(\det \circ f_0)$, without requiring weighted 
homogeneity, but the proof we have so far found does not permit it.  
\end{Remark}
\begin{proof}[Proof of \fullref{CorJacfor}] 
By assumption $(x, y) \mapsto (x, g(x, y))$ has an isolated singularity at 
$0$.  Hence, if $g_0(y) = g(0, y)$, then $g_0$ has an isolated singularity at 
$0$ and $\mu_a(f_0) \,  + \, \mu_{a, b}(f_0) = \mu(g_0)$.  By 
\fullref{ThmAFD}, $\mu_{\cE_2^{sy}}(f_0) = \dim_{\C} N \cK_{H, e} f_0$.  We 
will show that there is a surjective projection $N \cK_{H, e} f_0 \to 
\cO_{\C^{n-1}, 0}/ \Jac(g_0)$ with kernel the vector space in the RHS of 
\eqref{Eqn5.4}.  Then, by \fullref{Thmsym2} and the above remark, the 
result follows.

For $H$ the defining equation for $\cE_2^{sy}$, $\dlog (H)$ is generated by 
$\zeta_1 = a \pd{ }{b} + 2b \pd{ }{c}$ and $\zeta_2 = 2a \pd{ }{a} - b \pd{ 
}{b} - 4c \pd{ }{c}$.  Then, if instead we write $f_0(x, y) = (x, g(x, y), h(x, 
y))$, we obtain the generators for $T \cK_{H, e} f_0$ as an  $\cO_{\C^n, 
0}$--module
$$ \pd{f_0}{x} = (1, g_x, h_x) \qquad \makebox{ and } \qquad \pd{f_0}{y_i} 
= (0, g_{y_i}, h_{y_i})  $$
and
$$  \zeta_1 \circ f_0 = (0, x, 2g) \qquad \makebox{ and } \qquad  \zeta_2 
\circ f_0 = (2x, -g, -4h) \, .  $$
We may choose for generators for $\theta(f_0)$: $\gevar_1^{\prime} = (1, 
g_x, h_x)$, $\gevar_2 = (0, 1, 0)$, and $\gevar_3 = (0, 0, 1)$.  By the 
above, $\gevar_1^{\prime} \in T \cK_{H, e} f_0$; hence the projection of 
$\theta(f_0)$ to $\cO_{\C^n, 0}\{ \gevar_2, \gevar_3\}$ maps $T \cK_{H, 
e} f_0$ onto $L  = \cO_{\C^n, 0}\{\eta_1, \eta_2 , \xi_i, 1 \leq i \leq n-1  
\}$ with kernel $\cO_{\C^n, 0}\{ \gevar_1^{\prime}\}$,
where 
$$\eta_1 = (x, 2g), \,\, \eta_2 = (-g -2x g_x, -4h -2xh_x),\,\,  \text{ and } 
\,\,  \xi_i = (g_{y_i}, h_{y_i}) . $$
Thus, $N \cK_{H, e} f_0 $ is mapped isomorphically to  $\cO_{\C^n, 0}\{ 
\gevar_2, \gevar_3\}/ L$.

Next, we want to further project $\cO_{\C^n, 0}\{ \gevar_2, \gevar_3\}$ 
onto $\cO_{\C^n, 0}\{ \gevar_2\}$.  First, by the weighted homogeneity of 
$g$, we replace $\eta_2$ by 
\begin{align*}
\eta_2^{\prime} \, &= \, \ell \eta_2 + 2\ell g_x \eta_1 + a_0 g_x  \eta_1 
+ \sum_{i = 1}^{N-1} a_i y_i \xi_i   \\ 
\intertext{ and upon expanding and rearranging terms using the Euler 
relation for $g$ }
&=  \left(0, -(2\ell + a_0) \pd{(x h - g^2)}{x} - \left((2\ell - a_0) h - a_0 x \pd{h}{x} 
-\sum_{i = 1}^{N-1} a_i y_i \pd{h}{y_i}\right)\right) \\
   &= \,\, \left(0, -(2\ell + a_0) \pd{(x h - g^2)}{x} - \gd(h)\right).
\end{align*}
Under the projection onto $\cO_{\C^n, 0}\{ \gevar_2\}$, $\eta_2^{\prime} 
\mapsto 0$, so $L$ maps to $\cO_{\C^n, 0}\{x, g_{y_i}, i = 1, \dots , n-1  
\}$.  Thus, 
$$\cO_{\C^n, 0}\{ \gevar_2, \gevar_3\}/ L \,\, \rightarrow \,\, \cO_{\C^n, 
0}/(x, g_{y_i}, i = 1, \dots , n-1) \,\, \simeq  \,\, \cO_{\C^{n-1}, 
0}/\Jac(g_0) $$
is a surjective homomorphism onto the Jacobian algebra of $g_0$, which 
has length $\mu(g_0)$.  

Hence, it is enough to show that the kernel of this projection has the 
required form.  Since $\{x, g_{y_i}, i = 1, \dots , n-1 \}$ is a regular 
sequence, the only relations between these elements are the trivial ones.
Thus, the kernel of the projection is generated by
\begin{equation}
\label{Eqn5.4a}
\begin{split}
 &\left(0, (2\ell + a_0)\pd{(x h - g^2)}{x} + 
\gd(h)\right), \,\, \left(0, x h_{y_i} - 2 g g_{y_i}\right) \quad 1 \leq i \leq n-1, \,  \\
&\makebox{ and } \quad \left(0, g_{y_i} h_{y_j} -g_{y_j} h_{y_i}\right),  \quad 1 \leq 
i, j \leq n-1 \, .
\end{split}
\end{equation} 
Then, $\det \circ f_0 = x h - g^2$ and, provided $n \geq 3$, the $3\times 
3$ minors of $df_0$ are the $2 \times 2$ determinants $g_{y_i} h_{y_j} -
g_{y_j} h_{y_i}$.  Thus, under the isomorphism 
$\cO_{\C^n, 0}\{\gevar_3\} \simeq \cO_{\C^n, 0}$, the generators in 
\eqref{Eqn5.4a} are mapped to the the generators of 
$\widetilde{\Jac}(\det \circ f_0) + \Jac(df_0)$.  Thus, the kernel of the 
projection is isomorphic to the RHS of \eqref{Eqn5.4}.

Lastly, we note that if $\det \circ f_0$ is weighted homogeneous for the 
same weights as $g$, then $\wt(h) = 2\ell -a_0$.  Thus, by Euler's 
formula $\gd(h) = 0$.
\end{proof}
As a second application of \fullref{Thmsym2}, in \S\ref{S:sec9} we 
will obtain a ``$\mu = \tau$''-type formula for generic corank 
$1$ maps defining $2 \times 2$ symmetric matrix singularities.
\subsection{$3 \times 3$ Symmetric Matrices} 
Next,  we consider  $\mu_{\cD_3^{sy}}$ and use coordinates for $\Sym_3$ 
given by
$ A = \begin{pmatrix}
a & b & c \\
b  & d & e \\
c & e & f 
\end{pmatrix} $.
By our earlier discussion,  $\cD_3^{sy} \subset \Sym_3$ has a free 
completion 
$\cE_3^{sy} = \pi^* \cE_2^{sy} \cup \cD_3^{sy}$, with $\cE_3^{sy}$ defined 
by $a\,(ad -b^2)\cdot\det (A) = 0$.  
Then, by \eqref{Eqn5.2}, it is sufficient to determine $\tilde \chi_{\pi^* 
\cE_2^{sy} \cap \cD_3^{sy}}$.  To apply the inductive procedure, we will 
use the auxiliary linear free divisors given by \fullref{Prop3.4}  
(which arise from subgroups of $B_3$).  We obtain the following formulas 
for singular Milnor numbers.
\begin{Proposition}
\label{Prop5.5} 
On the space of germs transverse off $0$ to the associated varieties for 
$V(\cQ_a)$,
\begin{equation}
\label{Eqn5.5}
  \mu_{\cQ_a}  \,\, = \,\, \mu_{b d\cdot \cQ_a} \, - \,(\mu_{d, bc(bf-2ce)}  
+ \mu_d) \, + \, (\mu_{d, c, bf} + \mu_{d, c}) \, - \, (\mu_{b, cd} + \mu_b) . 
\end{equation}
There is an analogous formula for $\mu_{\cQ_f}$ obtained from
\eqref{Eqn5.5} by 
composing $f_0$ with the permutation $(a, b, c, d, e, f) \mapsto (f, e, c, d, 
b , a)$.  

By \fullref{MetaThm1} there is an analog of \eqref{Eqn5.5} 
for the Milnor number $\mu_{\varphi, \cQ_a}$ on the ICIS 
$X = \varphi^{-1}(0)$ defined by $\varphi \co \C^n, 0 \to \C^p, 0$. 
\end{Proposition}
\begin{Remark}
The  RHS of \eqref{Eqn5.5} is computed as the alternating sum of lengths 
of four determinantal modules using Theorems \ref{ThmAFD} and 
\ref{ThmAFDICIS}.  By \fullref{Prop3.4} and \fullref{Rem3.6}, 
$V(b\cdot d\cdot \cQ_a)$ is an $H$--holonomic linear free divisor and 
$V(bc(bf-2ce))$, after changing coordinates $E = 2e$, is an 
$H$--holonomic linear free divisor for the $2 \times 2$ general matrix 
$\bigl( \begin{smallmatrix} b & c \\  E & f \end{smallmatrix}\bigr)$.  
\end{Remark}
\begin{proof}[Proof of \fullref{Prop5.5}]
As $V(b\cdot d\cdot \cQ_a)$ is an $H$--holonomic linear free divisor, 
$V(\cQ_a)$ has a free completion, so we may apply \fullref{Lem2.9} to 
obtain
\begin{equation}
\label{Eqn5.6a}
\mu_{\cQ_a}   \,\, = \,\, \mu_{bd\cdot \cQ_a}  \, - \, \mu_{bd} \, +  \, (-
1)^{n-1} \tilde \chi_{b d, \cQ_a} \, .
\end{equation}
Then, it is sufficient to compute $\tilde \chi_{b d, \cQ_a}$.  Then,
$$ V(b d, \cQ_a) \,\, = \,\, V(b, \cQ_a)  \cup V(d, \cQ_a)   \,\, = \,\, 
V(b, c d)  \cup V(d, b(bf-2ce)) \, . $$
Also, $V(b, c d)  \cap V(d, b(bf-2ce)) = V(b, d)$.  Hence, applying 
\fullref{Lem2.9}, we obtain
\begin{equation}
\label{Eqn5.6}
\tilde \chi_{b d, \cQ_a}   \,\, = \,\, (-1)^{n-2}\big( \mu_{b, c d}  \, + \, 
\mu_{d, b(bf-2ce)}  -  \mu_{b, d} \big) \, .
\end{equation}
Now, $V(bc(bf-2ce))$ is a linear free divisor for the $2 \times 2$ general 
matrices.  Thus, by the metaversion of \fullref{Lem2.9}
\begin{equation}
\label{Eqn5.6b}
\mu_{d, b(bf-2ce)}   \,\, = \,\,\mu_{d, bc(bf-2ce)}  \, - \, \mu_{d, c} \, - \, 
\mu_{d, c, b f} \, .
\end{equation}
Substituting \eqref{Eqn5.6b} and \eqref{Eqn5.6} into \eqref{Eqn5.6a} and 
replacing 
$$     \mu_{b d} \, - \, \mu_{b, d}  = \mu_{b} \, + \, \mu_{d}  $$
yields \eqref{Eqn5.5}.  
\end{proof}
Then, $\cE_3^{sy}$ and $\cD_2^{sy} \cup V(\cQ_f)$ are 
$H$--holonomic free divisors by \fullref{thm:towerhholo}, 
respectively \fullref{Prop3.4} and \fullref{Rem3.6}.  Thus, 
using the formula given in \fullref{Prop5.5}, we may compute the 
singular Milnor number $\mu_{\cD_3^{sy}}$ using the following theorem. 
\begin{Thm}
\label{Thmsym3}  
For the space of germs transverse to the associated varieties  for 
$\cE_3^{sy}$ off $0$, the singular Milnor number can be computed by 
\begin{equation}
\label{Eqn5.7}
\mu_{\cD_3^{sy}} \,\,  =  \,\,  \mu_{\cE_3^{sy}} - \mu_{\cD_2^{sy} \cup 
\cQ_f} \, + \, \mu_{\cQ_f} \, - \, \big((\mu_{a, \cQ_a}  +  \mu_a ) \, + \, ( 
\mu_{a, b, c \cdot d} + \mu_{a, b})\big)
\end{equation} 
where $\mu_{\cE_3^{sy}} = \cK_{\tilde B_3, e}$--$\codim$, where $\tilde 
B_3$ is the subgroup of $B_3$ preserving the defining equation for 
$\cE_3^{sy}$.

By \fullref{MetaThm1} there is an analog of \eqref{Eqn5.7} 
for the Milnor number $\mu_{\varphi, \cD_3^{sy}}$ on the ICIS $X = 
\varphi^{-1}(0)$ defined by $\varphi \co \C^n, 0 \to \C^p, 0$. 
\end{Thm}
\begin{Remark}
In the RHS of \eqref{Eqn5.7}, the first two terms are lengths of 
determinantal modules, $\mu_{\cQ_f}$ is computed by 
\fullref{Prop5.5}, and of the last two groups of pairs of terms, the first pair 
is computed using the meta-version of \eqref{Eqn5.5} and 
\fullref{ThmAFDICIS}, and the second is the length of a determinantal module 
by \fullref{ThmAFDICIS}.   
\end{Remark}
\begin{proof}[Proof of \fullref{Thmsym3}]
We may apply \fullref{Lem2.9} to \eqref{Eqn5.1} to obtain
\begin{equation}
\label{Eqn5.8a}
     \mu_{\cD_3^{sy}} \,\, =   \mu_{\cE_3^{sy}} \, - \,  \mu_{\cE_2^{sy}} \, 
+ \,  (-1)^{n-1} \tilde \chi_{\pi^* \cE_{2}^{sy} \cap \cD_3^{sy}} \, , 
\end{equation}
provided we can compute 
$\tilde \chi_{\pi^* \cE_{2}^{sy} \cap \cD_3^{sy}}$.   Then, as 
$\cE_{2}^{sy}$ is defined by $a\, (ad-b^2) = 0$,  
\begin{align}
\label{eqn5.8}
\pi^* \cE_{2}^{sy} \cap \cD_3^{sy} \,\, &= \,\,  (V(a) \cap \cD_3^{sy}) \, 
\cup \, (V(ad - b^2) \cap \cD_3^{sy})  \notag \\
 &= \,\,  V(a, \cQ_{a}) \, \cup \, V(ad - b^2, \cQ_{f}) \, .
\end{align}
Also,  $V(a, \cQ_{a}) \, \cap \, V(ad - b^2, \cQ_{f}) = V(a, b, c\cdot d)$.  
Thus, applying \fullref{Lem2.9}, we obtain
\begin{equation}
\label{Eqn5.9}
\tilde \chi_{\pi^* \cE_{2}^{sy} \cap \cD_3^{sy}}   \,\, = \,\, 
\tilde \chi_{a, \cQ_{a}}  \, + \,  \tilde \chi_{ad - b^2, \cQ_{f}}  -  
\tilde \chi_{a, b, c\cdot d}  \, .
\end{equation}
Also, by \fullref{Lem2.9} 
\begin{equation}
\label{Eqn5.10}
\mu_{\cQ_{f}} \,\, = \,\, \mu_{(ad - b^2)\cdot\cQ_{f}} \, - \, \mu_{ad - 
b^2} + (-1)^{n-1} \tilde \chi_{ad - b^2, \cQ_{f}} \, .
\end{equation}
Then, for \eqref{Eqn5.8a}, we can use \eqref{Eqn5.10} to substitute for 
$\tilde \chi_{ad - b^2, \cQ_{f}}$ in \eqref{Eqn5.9}.  Next we evaluate the 
vanishing singular Euler characteristics in terms of singular Milnor 
numbers; for example, $\tilde \chi_{a, \cQ_{a}} = (-1)^{n-2}\mu_{a, 
\cQ_{a}}$,  $\tilde \chi_{a, b, c\cdot d} = (-1)^{n-3}\mu_{a, b, c\cdot d}$, 
and $V(ad-b^2) = \cD_2^{sy}$ so $V((ad-b^2)\cdot \cQ_f) = \cD_2^{sy} \cup 
V(\cQ_f)$.  Lastly, by \fullref{Thmsym2} we replace 
\begin{equation}
\label{Eqn5.10b}
\mu_{\cE_{2}^{sy}} \, - \,  \mu_{ad - b^2} \,\, = \,\, \mu_{a} \, + \, 
\mu_{a, b} \, \, .
\end{equation}
This yields \eqref{Eqn5.7}.
\end{proof}

In \S\ref{S:sec9} we will also obtain a ``$\mu = \tau$''-type 
formula for generic corank $1$ maps defining $3 \times 3$ symmetric 
matrix singularities.  
\section{General Matrix Singularities}  
\label{S:sec6} 
By the results \cite[Theorem 7.1]{DP1} for general matrices, summarized 
in \S \ref{S:sec3}, together with \fullref{thm:towerhholo}, both 
$\cE_m$ in $M_{m, m}$, and $\cE_{m-1, m}$ in $M_{m-1, m}$ are 
$H$--holonomic linear free divisors.  Moreover, the determinant variety  
$\cD_m$ in $M_{m, m}$ and the generalized determinant variety $\cD_{m-
1, m}$ in $M_{m-1, m}$, which has defining equation
$\det\!\left(\hat A^{(m-1)}\right) 
= 0$, have free completions  given by 
\begin{align}
\label{Eqn6.1}
     \cE_m \,\, & =  \,\,  \pi^* \cE_{m-1, m} \cup \cD_m   \qquad \makebox{ 
and } \notag  \\
     \cE_{m-1, m} \,\,  &= \,\,   \pi^{\prime\, *} \cE_{m-1} \cup \cD_{m-1, 
m} \, , 
\end{align}
for the projections $\pi \co M_{m, m} \to M_{m-1, m}$ and
$\pi^{\prime} \co 
M_{m-1, m} \to M_{m-1, m-1}$.

We first use these free completions to compute the singular Milnor 
number $\mu_{\cD_2}$ for $\cD_2 \subset M_{2, 2}$.  
\subsection{$2 \times 2$  Matrices} 
We use coordinates  $ \begin{pmatrix}
a  &b \\
c  & d 
\end{pmatrix}$ on $M_{2, 2}$ and consider the modified 
Cholesky-type representation.  Then, by \cite[Theorem 7.1]{DP1},  the 
exceptional orbit variety $\cE_2$ is defined by  $a\, b \cdot(ad -bc) = 0$.  
We then have the following: 
\begin{Thm}
\label{Thmgen2}  
On the space of germs transverse off $0$ to the associated varieties for  
$\cE_2$,  
\begin{equation}
\label{Eqn6.3}
   \mu_{\cD_2}   =    \mu_{\cE_2} -  ((\mu_a   + \mu_{a, cb}) \, + \, (\mu_b   
+ \mu_{b, ad}) ) \, .
\end{equation}  
Here $\mu_{\cE_2} = \cK_{\tilde G_2, e}$--$\codim$ where $\tilde G_2$ is 
the subgroup of $B_2 \times C_2$ preserving the defining equation  $a\, b 
\cdot(ad -bc) = 0$.   By \fullref{MetaThm1} there is an 
analog of \eqref{Eqn6.3} for singular Milnor number $\mu_{\varphi, 
D_2}$ on an ICIS $X = \varphi^{-1}(0)$ defined by $\varphi \co \C^n, 0 \to 
\C^p, 0$.  
\end{Thm}
\begin{Remark}
Each pair $\mu_a + \mu_{a, cb}$ and $\mu_b + \mu_{b, ad}$ is computed 
as the length of a determinantal module by \fullref{ThmAFDICIS}.  
\end{Remark}
As a corollary of the proof we obtain the following which will be used in 
the calculations for the skew-symmetric case.
\begin{Corollary}
\label{Corgen2}
With the assumptions of \fullref{Thmgen2},
\begin{equation}
\label{Eqn6.4a}
   \mu_{a(ad-bc)}   =    \mu_{\cE_2} -  (\mu_b   + \mu_{b, ad} ) \, 
\end{equation}
and 
\begin{equation}
\label{Eqn6.4b}
   \mu_{ad\,(ad-bc)}   =    \mu_{\cE_2} +  ((\mu_d   + \mu_{d, abc}) \, - \, 
(\mu_b   + \mu_{b, ad}) ) \, .
\end{equation}
There are also corresponding meta-versions of these formulas. 
\end{Corollary}
\begin{proof}[Proof of \fullref{Thmgen2} and \fullref{Corgen2}]
First, $\cD_2$ has the $H$--holonomic free completion $\cE_2$ defined 
by $ab\cdot (ad-bc)$.  Thus,  
\begin{equation}
\label{Eqn6.5a}
     \mu_{\cD_2} \,\, =   \mu_{\cE_2} \, - \,  \mu_{ab} \, + \, (-1)^{n-1} 
\tilde \chi_{ab, (ad-bc)}  \, .
\end{equation}
Since $V(ab, ad-bc) = V(a, bc) \cup V(b, ad)$ with $V(a, bc) \cap V(b, ad) = 
V(a, b)$, by \fullref{Lem2.9}
\begin{equation}
\label{Eqn6.6}
    \tilde \chi_{ab, (ad-bc)} \,\, = (-1)^{n-2} \big( \mu_{a,bc} \, + \,  
\mu_{b, ad} \, - \,  \mu_{a, b}  \big)  \, .
\end{equation}
Then, substituting \eqref{Eqn6.6} into \eqref{Eqn6.5a} and replacing 
$$  \mu_{ab} \, - \, \mu_{a, b} \,\, = \,\, \mu_{a} \, + \, \mu_{b}  $$
yields \eqref{Eqn6.3}.

For \fullref{Corgen2}, the argument for \eqref{Eqn6.4a} is similar 
using instead that $\cE_2$ is a free completion of $V(a(ad-bc))$.  While 
for 
\eqref{Eqn6.4b} we use 
$$V(ad\,(ad-bc)) = V(a (ad-bc)) \cup V(d) \quad \makebox{ with } \quad  
V(a (ad-bc)) \cap V(d) = V(d, abc) \, .$$
By \fullref{Lem2.9}
\begin{equation}
\label{Eqn6.7}
    \mu_{ad (ad-bc)} \,\, = \mu_{a (ad-bc)} \, + \,  \mu_{d} \, + \,  \mu_{d, 
abc} 
\end{equation}
and then we substitute \eqref{Eqn6.4a} for $\mu_{a (ad-bc)}$.
\end{proof}
As for symmetric matrices, we deduce in \S \ref{S:sec9} a ``$\mu = 
\tau$''-type formula for generic corank 1 germs for $2 \times 2$ 
general matrices.
\subsection{$2 \times 3$  Matrices} 
We use coordinates  $\begin{pmatrix}
a  & b & c\\
d  & e  & f
\end{pmatrix}$ on $M_{2, 3}$ and consider the modified 
Cholesky-type representation.  Again by \cite[Theorem 7.1]{DP1}, the 
exceptional orbit variety $\cE_{2, 3}$ is a free divisor and is defined by 
$a\, b\cdot (ae -bd)\cdot (bf -ce) = 0$.

We use this free divisor to compute $\mu_V$ where $V = V((ae -bd)\cdot 
(bf -ce))$.  To simplify notation, we let $V_j$ denote the subvariety of 
$M_{2, 3}$ defined by the determinant of the submatrix obtained by 
deleting the $j$--th column.  Also, we denote the union $V_i \cup V_j$ by 
$V_{i\, j}$.  Then, $V((ae -bd)\cdot (bf -ce)) = V_{1\,3}$.  Once we have 
computed $\mu_V$ for $V = V_{1\,3}$, then we may compute $\mu_V$ for 
$V = V_{i\,j}$ by permuting the coordinates corresponding to the 
permutation of the columns sending $(1, 3)$ to $(i, j)$.  
\begin{Thm}
\label{Thmgen2.3}  
For the space of germs transverse to the associated varieties for $\cE_{2, 
3}$ off $0$,  
\begin{equation}
\label{Eqn6.4}
   \mu_{V_{1\,3}} \,\,  =  \,\,  \mu_{\cE_{2, 3}} \, - \, (\mu_{a, bde (bf - 
ce)} + \mu_{a})\, + \, (\mu_{a, e, bdf}+ \mu_{a, e})\, - \, (\mu_{b, ace} + 
\mu_{b}) \, .
\end{equation}  
Here $\mu_{\cE_{2, 3}} = \cK_{\tilde G_{3}, e}$--$\codim$ where $\tilde 
G_{3}$ is the subgroup of $B_2 \times C_3$ which preserves the defining 
equation $a b\cdot (ae -bd)\cdot (bf -ce) = 0$.

By \fullref{MetaThm1}, there is an analog of \eqref{Eqn6.4} 
for singular Milnor number $\mu_{\varphi, V_{1\,3}}$ on the ICIS $X = 
\varphi^{-1}(0)$ defined  by $\varphi \co \C^n, 0 \to \C^p, 0$.  
\end{Thm}
\begin{Remark}
Each grouped pair on the RHS of \eqref{Eqn6.4} can be computed using 
\fullref{ThmAFDICIS} for AFD's on an ICIS and the first term by   
\fullref{ThmAFD}. Thus, the RHS of \eqref{Eqn6.4} is computed as the 
alternating sum of the lengths of four determinantal modules.

We can obtain the corresponding formulas for $\mu_{V_{1\,2}}$, resp. 
$\mu_{V_{2\,3}}$ by applying \eqref{Eqn6.4} after first composing $f_0$ 
with the permutation $(a, b, c, d, e, f) \mapsto (a, c, b, d, f, e)$, 
respectively,  $(a, b, c, d, e, f) \mapsto (b, a, c, e, d, f)$.  
\end{Remark}
\begin{proof}[Proof of \fullref{Thmgen2.3}]
First, $V((ae-bd)(bf-ce))$ has as a free completion 
$\cE_{2, 3} = V(ab(ae-bd)(bf-ce))$.  By \fullref{Lem2.9},
\begin{equation}
\label{Eqn6.8}
\mu_{V_{1\,3}}\,\, =\,\, \mu_{\cE_{2, 3}} \, - \, \mu_{ab}\, + \, (-1)^{n-1} 
\tilde \chi_{ab,(ae-bd)(bf-ce)}.
\end{equation}
Since $V(ab,(ae-bd)(bf-ce)) = V(a,bd(bf-ce))\cup V(b,ace)$ and $V(a,bd(bf-
ce))\cap V(b,ace) = V(a, b)$, we have by \fullref{Lem2.9} (by 
evaluating the $\tilde \chi$ as singular Milnor numbers), 
\begin{equation}
\label{Eqn6.9}
\tilde \chi_{ab,(ae-bd)(bf-ce)} \, \,  = \,\,  (-1)^{n-2}\big( \mu_{a,bd(bf-
ce)} \, + \, \mu_{b,ace}\, - \, \mu_{a,b} \big)  \, .
\end{equation}
Then, $V(bd(bf-ce))$ has a free completion $V(ebd (bf-ce))$.  Thus by the 
meta-version of \fullref{Lem2.9}, 
\begin{equation}
\label{Eqn6.10}
\mu_{a, bd(bf-ce)} \, \,  = \,\,  \mu_{a, bde(bf-ce)}\, - \, \mu_{a, e}\, - \, 
\mu_{a, e, bdf} \, .
\end{equation}
Then, by substituting \eqref{Eqn6.10} for $\mu_{a,bd(bf-ce)}$ into 
\eqref{Eqn6.9}, then substituting the resulting expression into 
\eqref{Eqn6.8}, and lastly replacing 
$$  \mu_{ab}  \, - \, \mu_{a,b}  \, \,  = \,\, \mu_{a}  \, + \, \mu_{b} \, , $$    
we obtain the result.
\end{proof}

\begin{Remark}  We have also obtained a formula for $3 \times 3$ general 
matrix singularities; however, we are not including it in this paper.
\end{Remark}
\section{Vanishing Topology for $2\times 3$ Cohen--Macaulay Singularities in $\C^n$}
\label{S:sec7}
In this section we apply the preceding results in reverse to obtain a
formula for the singular vanishing Euler characteristic for 
Cohen--Macaulay singularities in $\C^n$ defined by $2 \times 3$ matrices.  
These are given as $\cV_0 = f_0^{-1}(\cV)$, where $\cV$ is the variety of 
singular matrices of rank $\leq 1$ in $M_{2, 3}$ and $f_0 \co \C^n, 0 \to 
M_{2, 3}, 0$ is transverse to $\cV$ off $0$.  We then apply this formula in 
several different ways.  First, if $n = 4, 5$ or $6$, then $\cV_0$ will be 
an isolated surface, resp.\ 3--fold, resp.\ 4--fold, singularity.  In the case 
of $n = 4$, we obtain a formula for the Milnor number for isolated 
$2\times 3$
Cohen--Macaulay surface singularities as the sum of lengths of determinantal 
modules.  Furthermore in the case of the $2\times 3$ Cohen--Macaulay 3--fold 
singularities, we obtain a formula for the difference of the second and 
third Betti numbers $b_3 - b_2$ of the Milnor fiber.  We furthermore 
deduce bounds on these Betti numbers.  In \S \ref{S:sec9}, we shall 
implement these formulas using the results of \S \ref{S:sec6}, with a 
software package developed for Macaulay2, to compute the Milnor number 
for simple 
$2\times 3$
Cohen--Macaulay surface singularities and $b_3 - b_2$ for 3--fold 
singularities.  

In addition, if we consider instead $2\times 3$ Cohen--Macaulay singularities on an 
ICIS $X$ defined by $\varphi$, then we obtain analogous results in each 
case  using the corresponding meta-versions of the results.  Finally, we 
also use these results to obtain formulas for the Milnor numbers of 
functions defining ICIS on isolated $2\times 3$ Cohen--Macaulay singularities.  
\subsection{Singular Vanishing Euler Characteristic for Nonisolated 
$2\times 3$
Cohen--Macaulay Singularities in $\C^n$}
Let $M_{2, 3}$ denote the space of $2 \times 3$ matrices with $\cV$ the 
variety of singular matrices of rank $\leq 1$.  Consider $f_0 \co \C^n, 0 \to 
M_{2, 3}, 0$.  Because $\cV$ is not a complete intersection, $f_0$ does 
not have a singular Milnor number $\mu_{\cV}(f_0)$.  However, we can use 
\fullref{Prop2.14} to compute $\tilde \chi_{\cV}(f_0)$.  
\begin{Thm}
\label{ThmVanchi}  For a germ $f_0 \co \C^n, 0 \to M_{2, 3}, 0$ which is 
transverse to the associated varieties off $0$, let $\cV_0 = f_0^{-
1}(\cV)$ be the nonisolated Cohen--Macaulay singularity. Then, the 
singular vanishing Euler characteristic is computed by 
\begin{equation}
\label{Eqn7.5}
\tilde \chi_{\cV}(f_0) \,\,  =  \,\,  (-1)^{n-1} \left(\mu_{V_{1\, 2\, 
3}}(f_0) \, - \, \sum \mu_{V_{i\, j}}(f_0) + \sum_{i = 1}^{3} 
\mu_{V_i}(f_0)  \right)
\end{equation}
where the first sum is over $\{i, j\} = \{ 1, 2\}, \{ 1, 3\}, \{ 2, 3\}$ and 
$V_{1\, 2\, 3} = V_1 \cup V_2  \cup V_3$.  

By \fullref{MetaThm1} there is an analog of \eqref{Eqn7.5} 
for vanishing Euler characteristic $\tilde \chi_{\varphi, D_2}$ on the ICIS 
$X = \varphi^{-1}(0)$ defined by $\varphi \co \C^n, 0 \to \C^p, 0$.  
\end{Thm}
\begin{Remark}
\label{Rem7.5}
Here we are using the notation of \S \ref{S:sec6}.  The $\mu_{V_{i\, j}}$ 
are computed by \fullref{Thmgen2.3}, and the $\mu_{V_i}$ are 
computed by \fullref{Thmgen2}.  Also, as explained in \S 
\ref{S:sec3},  the variety $V_{1\, 2\, 3}$ is an $H$--holonomic linear free 
divisor corresponding to a quiver representation by 
Buchweitz--Mond \cite{BM}.  Hence, $\mu_{V_{1\, 2\, 3}}$ can be computed 
as the length of a determinantal module by \fullref{ThmAFD}. 
\end{Remark}
As we will see in \S\ref{S:sec9}, we can frequently apply generic 
reduction  by applying an element of $\GL_2(\C) \times \GL_3(\C)$ to 
$f_0$ so that, depending on rank of $df_0(0)$, the terms in \eqref{Eqn7.5} 
either vanish or their computation considerably simplifies.  
\subsection{Milnor Numbers for Isolated $2\times 3$ Cohen--Macaulay Surface 
Singularities in $\C^4$} 
We now consider the special case of $f_0 \co \C^4, 0 \to M_{2, 3}, 0$ which 
is transverse to $\cV$ off $0$.  By the Hilbert--Burch Theorem, $\cV_0 = 
f_0^{-1}(\cV)$ is an isolated Cohen--Macaulay surface singularity.  By 
results of Wahl \cite{Wa} (in the weighted homogeneous case) and 
Greuel--Steenbrink \cite{GS}, its Milnor fiber has first Betti number $b_1 
= 0$.  By convention, the second Betti number is referred to as the Milnor 
number $\mu(\cV_0)$.

In this case, the versal unfolding of $\cV_0$ in the sense of algebraic 
geometry is obtained by a deformation of the mapping $f_0$, see 
\cite{Sh}. Thus, what we call the singular Milnor fiber is actually the 
Milnor fiber of $\cV_0$ since a stabilization of $f_0$ will only 
(transversely) intersect the smooth part of $\cV$.  Hence, we may 
compute $\mu(\cV_0) = \tilde \chi_{\cV}(f_0)$.  By applying an element 
of $\GL_2(\C) \times \GL_3(\C)$ to $f_0$ we may assume that $f_0$ is 
transverse to all of the associated varieties for each $V_i$ and $V_{i\, 
j}$.  Then, the preceding results yield the following formula for 
$\mu(\cV_0)$.
\begin{Thm}
\label{ThmCM}  
For a germ $f_0 \co \C^4, 0 \to M_{2, 3}, 0$ which is transverse to the 
associated varieties off $0$, let $\cV_0 = f_0^{-1}(\cV)$ be the isolated 
Cohen--Macaulay surface singularity. Then, the Milnor number is computed 
by 
\begin{equation}
\label{Eqn7.6}
\mu(\cV_0) \,\,  =  \,\,  \sum \mu_{V_{i\, j}}(f_0) -  \sum_{i = 1}^{3} 
\mu_{V_i}(f_0) \, - \, \mu_{V_{1\, 2\, 3}}(f_0) 
\end{equation}
where the first sum is over $\{i, j\} = \{ 1, 2\}, \{ 1, 3\}, \{ 2, 3\}$.  
By \fullref{MetaThm1} there is an analog of \eqref{Eqn7.6} 
for the Milnor number $\mu(\cV_0)$ on the ICIS $X = \varphi^{-1}(0)$ 
defined by $\varphi \co \C^n, 0 \to \C^{n-4}, 0$. 
\end{Thm}
All of \fullref{Rem7.5} applies equally well to \fullref{ThmCM}.  
\subsection{Betti Numbers of Milnor Fibers for Isolated $2\times 3$ Cohen--Macaulay 
3--fold Singularities in $\C^5$}
We consider the case  $f_0 \co \C^5, 0 \to M_{2, 3}, 0$ which is transverse 
to $\cV$ off $0$.  Now $\cV_0 = f_0^{-1}(\cV)$ is an isolated 
Cohen--Macaulay 3--fold singularity.  A stabilization of $f_0$ will 
miss the isolated singular point $0 \in \cV$; hence the singular Milnor 
fiber for $f_0$ is the Milnor fiber of $\cV_0$.  Thus, the singular 
vanishing Euler characteristic of $f_0$ is the vanishing Euler 
characteristic of $\cV_0$.  The results of Greuel--Steenbrink still apply; 
and so the first Betti number $b_1(\cV_0)  = 0$ (in fact, they show that 
the Milnor fiber of $\cV_0$ is simply connected).  Thus, $\tilde 
\chi_{\cV}(f_0)  = b_2(\cV_0) - b_3(\cV_0)$.  Then, we may compute this 
difference: 
\begin{Thm}
\label{Thmb2-b3}  For a germ $f_0 \co \C^5, 0 \to M_{2, 3}, 0$ which is 
transverse to the associated varieties off $0$, let $\cV_0 = f_0^{-
1}(\cV)$ be the isolated Cohen--Macaulay 3--fold singularity. Then, 
\begin{equation}
\label{Eqn7.5c}
b_3(\cV_0) - b_2(\cV_0) \,\,  =  \,\, \sum \mu_{V_{i\, j}}(f_0) -  \sum_{i 
= 1}^{3} \mu_{V_i}(f_0) \, - \, \mu_{V_{1\, 2\, 3}}(f_0) 
\end{equation}
where the first sum is over $\{i, j\} = \{ 1, 2\}, \{ 1, 3\}, \{ 2, 3\}$.

By \fullref{MetaThm1} there is an analog of 
\eqref{Eqn7.5c} for the difference $b_2(\cV_0 \cap X) - b_3(\cV_0 \cap 
X)$ on the ICIS $X = \varphi^{-1}(0)$ defined by $\varphi \co \C^n, 0 \to 
\C^{n-5}, 0$.  
\end{Thm}
There are analogous remarks as earlier regarding the computation of the 
RHS of \eqref{Eqn7.5c}.  Depending on the sign of the RHS of 
\eqref{Eqn7.5c}, it gives either a crude lower bound on $b_2(\cV_0)$ if the 
RHS is positive, or on $b_3(\cV_0)$ if the RHS is negative.  
\subsection{Milnor Numbers for Isolated ICIS singularities on Isolated 
$2\times 3$
Cohen--Macaulay Singularities}
As a final consequence of the meta-versions of the preceding results, we 
consider $\cV_0$ an isolated Cohen--Macaulay surface or 3--fold 
singularity defined by $f_0 \co \C^n, 0 \to M_{2, 3}, 0$ for $n = 4, 5$.  Also, 
let $\varphi \co \C^n, 0 \to \C^p,0 $ be an ICIS germ defining $X, 0 \subset 
\C^n, 0$, with $n-p \geq \dim \cV_0$, and so that $\varphi | \cV_0$ has an 
isolated singularity.  We let $X_0 = \varphi^{-1}(0) \cap \cV_0$ and 
consider the Milnor fiber $X_t$ of $\varphi | \cV_0$.  Then, $X_0$ is again 
an isolated Cohen--Macaulay (point, curve or surface) singularity.
We can use the preceding results to compute the Milnor number.
\begin{Corollary}
\label{Cor7.6}
In the preceding situation, the Milnor number of the restriction $\mu(X_0) 
= \chi_{\varphi, \cV}(f_0)$, which can be computed using the 
meta-version of \eqref{Eqn7.5} which becomes the meta-versions of either 
\eqref{Eqn7.6} or \eqref{Eqn7.5c}.
\end{Corollary}
\begin{proof}
We may construct stabilizations of $f = (\varphi, f_0)\co \C^n, 0 \to \C^p 
\times M_{2, 3}$ in two different ways: either by stabilizing $\varphi$ by 
$\varphi_t$ so the Milnor fiber $\varphi_t^{-1}(0)$ intersects $\cV_0$ 
transversely; or by stabilizing $f_0$ (as a nonlinear section of $\cV$) by 
$f_t$ so $\cV_t = f_t^{-1}(\cV)$ intersects $X$ transversely.  As both of 
these are stabilizations of the same germ $f$ as a nonlinear section of $\{ 
0\} \times \cV \subset \C^p \times M_{2, 3}$, the singular Milnor fibers 
are diffeomorphic, and hence, they have the same Euler characteristic.  
Thus, for the first, we obtain the Milnor number $\mu(X_0)$.  For the 
second, we have $\chi_{\varphi, \cV}(f_0)$, and the meta-version of 
\eqref{Eqn7.5} allows us to compute it.  This becomes the meta-version of 
either \eqref{Eqn7.6} or \eqref{Eqn7.5c}.  
\end{proof}
\section{Skew-Symmetric Matrix Singularities}  
\label{S:sec8} 
We use coordinates for $\Sk_4$ given by
$$  A \,\, = \,\,  \begin{pmatrix}
0   & a & b & c \\
-a  &0  & d & e \\
-b & -d & 0 & f \\
-c & -e & -f & 0
\end{pmatrix}  . $$
The determinantal variety $\cD_4^{sk}$ has reduced defining equation the 
Pfaffian $\Pf(A)$, which we shall denote simply as $\Pf$.  Then, by 
\cite[Theorem 8.1]{DP1} and also \cite[Theorem 5.2.21]{P}, the nonlinear 
solvable Lie algebra $\cL_4$ determines a free divisor $\cE_4^{sk}$, 
which is defined by $a\,b\, d\, (be -dc)\cdot\Pf (A) = 0$.  Also $a\,b\, d\, 
(be -dc) = 0$ defines a free divisor $\cE_2^{\prime}$ (the product union of 
$\{ 0\} \subset \C$ defined by $a = 0$ with $\cE_2$ for the $2 \times 2$ 
upper right-hand submatrix of $A$).  Hence, the Pfaffian hypersurface 
$\cD_4^{sk}$ has a free completion by this free divisor 
$$\cE_4^{sk} = \pi^* \cE_2^{\prime} \cup \cD_4^{sk} \, .$$
We denote $\pi^* \cE_2^{\prime}$ simply by $\cE_2^{\prime}$. We can also 
use this to give a free completion of $V((be -dc)\cdot\Pf (A))$. We next 
use this free completion to compute the singular Milnor number 
$\mu_{\cD_4^{sk}}$ via the following theorem.  
\begin{Thm}
\label{Thmsk4}  For the space of germs transverse to the associated 
varieties  for $\cE_4^{sk}$ off $0$, the singular Milnor number can be 
computed by 
\begin{equation}
\label{Eqn8.1b}
\mu_{\cD_4^{sk}} \,\,  =  \,\,  \mu_{\cE_4^{sk}} \, - \,  \mu_{a, f, (be-cd)}  
\, + \, \gl_1  \, + \, \gl_2  \, + \, \gl_3 \, ,
\end{equation} 
where each $\gl_k$ is a sum of terms of defining codimension $k$ and are 
given by
\begin{equation}
\label{Eqn8.2b}
\begin{split}
\gl_1 \, &= \, - \big(\mu_{b, cd\,(af+cd)}  + \mu_{d, be\,(af-be)} + 2 
\mu_{a, (be-cd)}\, + \, \mu_{f, (be-cd)}\big) \\
\gl_2 \,  &= - \, \big(\mu_{be-cd}\, + \,  \mu_{a, b, c \cdot d}  \, + \, 
\mu_{a, d, b \cdot e}\big)    \\
\gl_3\,  &=  \, (\mu_{a, b, d} + \mu_{b, d}) - \mu_{abd} \, .
\end{split}
\end{equation}
Here $\mu_{\cE_4^{sk}} = \cK_{\tilde \cL_4, e}$--$\codim$, where $\tilde 
\cL_4,$ is the Lie subalgebra of $\cL_4,$ preserving the defining equation 
for $\cE_4^{sk}$.

By \fullref{MetaThm1} there is an analog of 
\eqref{Eqn8.1b} (and \eqref{Eqn8.2b}) for the Milnor number $\mu_{\varphi, 
\cD_4^{sk}}$ on the ICIS $X = \varphi^{-1}(0)$ defined by $\varphi \co \C^n, 0 
\to \C^p, 0$. 
\end{Thm}
Also, the terms in the $\gl_i$ can be computed using the meta-versions of 
\fullref{Thmgen2} and \fullref{Corgen2}.
\begin{proof}
We first consider $V((be-cd)\cdot \Pf)$.  By \fullref{Lem2.9}
\begin{equation}
\label{Eqn8.3}
\mu_{\Pf} \,\, = \,\, \mu_{(be-cd)\cdot \Pf} \, -\, \mu_{be-cd} \, + \, (-
1)^{n-1} \tilde \chi_{be-cd, \Pf} \, .
\end{equation}
As $\cE_4^{sk}$ as a free completion of $V((be-cd)\cdot \Pf)$, by 
\fullref{Lem2.9}
\begin{equation}
\label{Eqn8.3a}
\mu_{(be-cd)\cdot \Pf} \,\, = \,\, \mu_{\cE_4^{sk}} \, -\, \mu_{abd} \, + \, 
(-1)^{n-1} \tilde \chi_{abd, (be-cd)\cdot \Pf} \, .
\end{equation}
Next, to compute $ \tilde \chi_{be-cd, \Pf}$ we observe
$$ V(be-cd, \Pf) \,\, = \,\, V(be-cd, af) \,\, = \,\,  V(a, be-cd) \cup V(f, 
be-cd)  $$
and $V(a, be-cd) \cap V(f, be-cd) = V(a, f, be-cd)$.  Hence, by 
\fullref{Lem2.9}
\begin{align}
\label{Eqn8.4}
\tilde \chi_{be-cd, \Pf} \,\, &= \,\, \tilde \chi_{a, be-cd} \, +\, \tilde 
\chi_{f, be-cd} \, - \tilde \chi_{a, f, be-cd}  \notag  \\
&= \,\, (-1)^{n-2} \big( \mu_{a, be-cd} \, +\, \mu_{f, be-cd} \, + \mu_{a, f, 
be-cd} \big) \, .
\end{align}
Lastly, we consider $\tilde \chi_{abd, (be-cd)\cdot \Pf}$.  
Observe that
$$ V(abd, (be - cd)\cdot \Pf)\, \, = \,\, V(a, (be-cd)) \cup V(b, cd(af +cd)) 
\cup V(d, be(af - be)) \, . $$
In addition, 
\begin{equation}
\label{Eqn8.5}
\begin{split}
V(a, (be-cd)) \cap V(b, cd(af +cd)) \, \, &=  \,\, V(a, b, cd)  \\
V(a, (be-cd)) \cap V(d, be(af - be)) \, \, &=  \,\, V(a, d, be)   \\ 
V(b, cd(af +cd)) \cap V(d, be(af - be)) \, \, &=  \,\, V(b, d) \, ;
\end{split}
\end{equation}
and 
\begin{equation}
\label{Eqn8.6} 
V(a, (be-cd)) \cap V(b, cd(af +cd)) \cap V(d, be(af - be)) \, = \, V(a, b, d) \, 
.
\end{equation}

Thus, since all of the terms on the RHS of \eqref{Eqn8.5} and \eqref{Eqn8.6} 
will define AFD's on ICIS, we may apply \eqref{Eqn2.13}
and evaluate 
each $\tilde \chi$ as a singular Milnor number to obtain
\begin{multline}
\label{Eqn8.7} 
\tilde \chi_{abd, (be-cd)\cdot \Pf} \,\, = \,\, (-1)^{n-2}\big( \mu_{a, be-
cd}  + \mu_{b, cd(af +cd)}   + \mu_{d, be(af - be)}\big)  \\
\, - \, (-1)^{n-3}\big( \mu_{a, b, cd}  + \mu_{a, d, be} - \mu_{b, d}\big) \, + 
\,  (-1)^{n-3}\mu_{a, b, d} \, .
\end{multline}

Finally, we substitute \eqref{Eqn8.7} into \eqref{Eqn8.3a}, and substitute 
the resulting \eqref{Eqn8.3a} and \eqref{Eqn8.4} into \eqref{Eqn8.3}.  After 
rearranging terms and simplifying coefficients we obtain \eqref{Eqn8.1b}.
\end{proof}
\begin{Remark}
Because there are several ways to give a free completion for 
$\cD_4^{sk}$, there are several variations on the formulas given in 
\fullref{Thmsk4} (see e.g. Theorem 6.2.11 of \cite{P}).  We have given 
a version which is conceptually shortest in terms of having to compute 
the fewest number of singular Milnor numbers in \eqref{Eqn8.1b}. 
\end{Remark}

For generic corank $1$ skew-symmetric matrix singularities, it will 
follow by generic reduction that all of the $\gl_i$ for $i > 0$ in 
\eqref{Eqn8.1b} vanish. In \S \ref{S:sec9} we further compute the two 
remaining terms and will obtain a ``$\mu = \tau$''-type result.
\section{Higher Multiplicities of Linear Free Divisors} 
\label{S:sec8.5}  
We will begin computing the general formulas in the special cases of 
mappings $f_0$ within restricted classes with a goal of relating 
$\mu_{\cD}(f_0)$ for $\cD$ a determinantal variety and $\tau = \cK_{M, 
e}$--$\codim(f_0)$.  For this we must first compute $\mu_{\cE}(f_0)$ for 
various $H$--holonomic free divisors $\cE$ and then apply the results of 
the previous sections.
 
We begin with the simplest case where $f_0$ is a generic linear section.  
Then, we are really computing the higher multiplicities for 
($H$--holonomic) linear free divisors.  We recall that for a hypersurface 
(or more generally a complete intersection) $\cV, 0 \subset \C^N, 0$ we 
may define for $0 < k < N$ the $k$--th higher multiplicity, denoted 
$\mu_k(\cV)$, as the singular Milnor number $\mu_{\cV}(\iti)$ for a 
generic linear section $\iti \co \C^k, 0 \to \C^N, 0$.  This is analogous to the 
definition of Teissier's $\mu_{*}$ sequence for isolated hypersurface 
singularities \cite{Te} and \cite{LeT}.  To be consistent with our earlier 
notation, if $k < \ell = \codim \cV$, then we let $\mu_k(\cV) 
\overset{\rmdef}{=} (-1)^{k-\ell + 1}$.  If $\cV$ is a hypersurface then 
$\mu_0(\cV) = 1$.

Very surprisingly, in the case of $H$--holonomic linear free divisors, 
these higher multiplicities can be computed independent of the specific 
linear free divisor $\cV$.
\begin{Proposition}
\label{Prophighmult}
If $\cV, 0 \subset \C^N, 0$ is an $H$--holonomic linear free divisor, then 
\begin{equation}
\label{Eqnhighmult}
\mu_k(\cV) \,\, = \,\, \binom{N-1}{k} \qquad   0 < k < N \, .
\end{equation}
Hence, for any $H$--holonomic linear free divisor in $\C^N$, there is the 
duality relation $$ \mu_k(\cV) \,\, =\,\,  \mu_{N-1-k}(\cV) \qquad 0 \leq 
k \leq N-1 .  $$
\end{Proposition}

Before proving the proposition, we point out as a consequence that any 
two $H$--holonomic linear free divisors in $\C^N$ will always have the 
same higher multiplicities.  Hence, it follows they all have a complex link 
which is a real homotopy $(N-1)$--sphere.  
\begin{Example}
\label{Examhighmult}
There are three exceptional orbit varieties in $M_{2, 3}$: that for the 
action of the solvable group $B_2 \times C_3$  given by modified Cholesky 
factorization; the ``quiver discriminant'' arising from the 
reductive group $(\GL_3 \times (\C^*)^3)/ \C^*$ for the quiver 
representation just mentioned; and that for $(\C^*)^6$ given by the 
coordinate hyperplane arrangement.  These are quite distinct 
$H$--holonomic linear free divisors in $M_{2, 3}$.  However, by 
\fullref{Prophighmult}, the $k$--th higher multiplicities for them 
all equal $\tbinom{5}{k}$.  
\end{Example}
We thus obtain the higher multiplicities for the linear free divisors listed 
in \fullref{table3.0}.
\begin{Proposition}
\label{Proplistmult}
For the free divisors in \fullref{table3.0}, the corresponding higher 
multiplicities $\mu_k$ are given by \fullref{table4.0}.  
\end{Proposition}
\begin{table}
\begin{tabular}{r||c|c|c|c}
Free     &   $\cE_m^{sy}$   & $\cE_m$  & $\cE_{m-1, m}$ &   $\cE_m^{sk}$ 
\\
  Divisor               &     &      &     &  \\
\hline
$\mu_k$ &  $\binom{\binom{m + 1}{2} - 1}{k}$  & $\binom{m^2-1}{k}$  &   
$\binom{m(m-1)-1}{k}$  &  $\gs_k\!\left(1^{\binom{m}{2} - (m-2)}, 2, 2, \dots , 
[(m+1)/2]\right)$ \\
\end{tabular}
\caption{\label{table4.0} Higher multiplicities for the exceptional orbit 
varieties $\cE$ for the solvable group and solvable Lie algebra block 
representations in \fullref{table2.0}. See \fullref{table3.0}.}
\end{table}
In the table, $\gs_k$ denotes the $k$--th elementary symmetric function, 
and $1^{\ell}$ denotes $1$ being repeated $\ell$ times and $2, 2, \dots, 
[(m+1)/2]$ denotes the sequence of $m-3$ integers $2, 2, 3, 3, \dots$, 
truncated at $[(m+1)/2]$.
\begin{Remark}
\label{Rem10.6}
We note that in the table $\cE^{sy}_3$, $\cE_{2, 3}$ and $\cE^{sk}_4$ are 
linear free divisors in $\C^6$; but $\cE^{sk}_4$ will have different higher 
multiplicities because it is not a linear free divisor.  In fact the values 
$\gs_k(1^4, 2) = 6, 14, 16, 9, 2$ for $ k = 1, \dots , 5$ also do not satisfy 
the duality property in \fullref{Prophighmult}.  Surprisingly,  the 
higher multiplicities $\mu_k(\cD^{sy}_2)$, $\mu_k(\cD^{sy}_3)$, 
$\mu_k(\cD_2)$, and $\mu_k(\cD^{sk}_4)$ do satisfy the duality property. 
This follows by the calculations in \S\S\,\ref{S:symmatr}, \ref{S:sec6} 
and \ref{S:sec8}.  For $\cD^{sy}_2$, $\cD_2$ and $\cD^{sk}_4$ it also 
follows because their defining equations have Morse singularities at $0$, 
and the restrictions to a generic section are again Morse singularities and 
their Milnor fiber is the singular Milnor fiber of the generic section.  Thus, 
all of the nonzero higher multiplicities equal $1$.  By contrast the higher 
multiplicities  $\mu_k(\cD^{sy}_3) = $ 1, 2, 4, 4, 2, 1 for $k  = 0, 1, \dots, 
5$ still satisfy the duality property.   This leads to:

{\bf Question/Conjecture}\stdspace
{\it
The higher multiplicities for the determinantal 
varieties $\cD^{sy}_n$ and $\cD_n$ satisfy the duality property.  
}

Because duality does not hold for $\cE^{sk}_4$, it suggests that the result 
for $\cD^{sk}_4$ may only be a low dimension phenomenon.  
\end{Remark}
\begin{proof}[Proof of Propositions \ref{Prophighmult} and \ref{Proplistmult}]
Both propositions are a consequence of the fact that for all such free 
divisors $\cV$, the module $N\cK_{\cV, e}\cdot \iti$ is (weighted) 
homogeneous in the sense of \cite{D5}; hence by Theorem 1 of \cite{D5} 
its length is given by a formula in terms of its weights.  This will yield 
the result.

The weighted homogeneous case for $N\cK_{\cV, e}\cdot f_0$, concerns 
$f_0 \co \C^n, 0 \to \C^N, 0$ with $\cV$ a free divisor such that we can 
choose weights for $\C^n$ and $\C^N$ so that: i)  both $f_0$ and $\cV$ are 
weighted homogeneous for the same weights; and ii) the generators of 
$\dlog(H)$ may also be chosen to be weighted homogeneous for these 
weights.  In our cases, we use weights $0$ for the coordinates of $\C^N$  
and $1$ for the weights of the coordinates $x_j$ for $\C^n$.  Then, as the 
section $\iti$ is linear, $\pd{ \iti}{x_j}$ has weight $0$ and for linear 
free divisors, $\zeta_j \circ \iti$ has weight $1$, while for $\cE^{sk}_m$ 
the last $m-3$ generators will have weights $2, 2, 3, 3, \dots$ as in the 
statement.  Then, by Theorem 1 of \cite{D5}, $\mu_k(\cE) = 
\mu_{\cE}(\iti)$ will equal $\gs_k(1, \dots , 1)$ with $(N-1)$ $1$'s  
($= \tbinom{N-1}{k}$) for a linear free divisor $\cE$, or $\gs_k(1, \dots, 
1, 2, 2, \dots , [(m+1)/2])$ with $(\tbinom{m}{2} - (m-2))$ $1$'s in the 
case of $\cE = \cE^{sk}_m$ (and $N = \tbinom{m}{2}$).   
\end{proof}
We use the preceding propositions in conjunction with two other 
properties of higher multiplicities which follow from 
\fullref{ProptysingMil}.
\begin{Proposition}
\label{Proptyhighmult}
Let $\cV, 0 \subset \C^N, 0$ be an $H$--holonomic free divisor.
\begin{enumerate}
\item
\label{enit:Proptyhighmult1}
If $\cV^{\prime} = \cV \times \C^p, 0 \subset \C^{N+p}, 0$, then 
$$ \mu_k (\cV^{\prime}) \,\, =  \,\, \mu_k (\cV) \qquad \makebox{ for }  0 
\leq k < N  \, .  $$
\item
\label{enit:Proptyhighmult2}
If $\cV^{\prime \prime}, 0 = \cV \times \{ 0\} \subset \C^{N+ p}, 0$ 
is the image of $\cV, 0$ via the inclusion $\C^N, 0 \subset \C^{N+ p}, 0$ 
(so that $\cV^{\prime \prime}$ is a free divisor in a linear subspace of 
$\C^{N+ p}$), then 
$$ \mu_k (\cV^{\prime\prime}) \,\, =  \,\, \mu_{k-p} (\cV) \quad 
\makebox{if $k \geq p$, and  $\,\, = (-1)^{p-k}$ if $k < p$} \, .     $$ 
\end{enumerate}
\end{Proposition}
\begin{proof} 
For \eqref{enit:Proptyhighmult1},
we can choose a generic linear section $\iti \co \C^k, 0 \to \C^{N+p}$ 
of $\cV^{\prime}$ so that $\pi \circ \iti$ is also a generic linear section 
of $\cV$ and the result follows from \eqref{enit:ProptysingMil1} of 
\fullref{ProptysingMil}.

For \eqref{enit:Proptyhighmult2},
provided $k \geq p$, we may choose a generic linear section
$\iti \co 
\C^k, 0 \to \C^{N+p}$ so that $\iti$ is transverse to $\C^p$ and if $W = 
\iti^{-1}(0) \times \C^N$ then $\pi \circ \iti\, | W$ is a generic linear 
section of $\cV$.  Then, \eqref{enit:Proptyhighmult2} follows by applying
\eqref{enit:ProptysingMil2}
of 
\fullref{ProptysingMil}.  
\end{proof}
\section{$\mu = \tau - \gg$-type Results for Matrix Singularities}  
\label{S:sec9} 
In this section we consider the relation between $\mu$  and $\tau$ for 
singularities defined by $f_0$.  Here $\mu$ will denote a singular Milnor 
number $\mu_{\cV}(f_0)$ or possibly the Milnor number of a 
Cohen--Macaulay isolated surface singularity, and 
$\tau$ will denote an appropriate  $\cK_{H, e}$--codimension of $f_0$.  We will 
be concerned with how much $\mu$ differs from $\tau$ or equivalently 
consider the difference $\gg = \tau - \mu$.  We recall the results for an 
ICIS $X, 0$ with $\mu$ the usual Milnor number and $\tau$ the Tjurina 
number (which is also the $\cK_{e}$--codimension).  Greuel showed that $\mu = 
\tau$ when $X$ is weighted homogeneous (see \cite{Gr} or  \cite[Chap. 
9]{L}); and Looijenga--Steenbrink showed that $\mu \geq \tau$ in general 
\cite{LSt}.  Thus, for ICIS, $\gg \leq 0$.  An analogous result was shown to 
hold for the ``discriminant Milnor number'' in \cite{DM}.  For 
matrix singularities, we consider what form such a result takes.  We will 
show for matrix singularities which are hypersurfaces defined by corank 
$1$ mappings that $\gg = 0$.  However, when we consider 
Cohen--Macaulay singularities defined from $2 \times 3$ matrices there are some 
fundamental changes which occur and $\gg$ becomes positive.
\subsection{Corank 1 mappings and $\mu = \tau$-type Results}
We begin by considering matrix singularities defined by corank $1$ 
mappings $f_0 \co \C^n, 0 \to M, 0$ of finite $\cK_M$--codimension for 
various spaces of matrices $M$ (with $\dim M = N$).  Here corank refers to 
the corank of $df_0(0)$ and not that of the specific matrices $f_0(x)$.  

As a prelude, we first consider germs $f_0 \co \C^n, 0 \to \C^N, 0$ with $n 
\geq N$ and $\cV \subset \C^N$ an $H$--holonomic linear free divisor.  We 
consider such corank $1$ mappings which are generic, in the sense that $W 
= df_0(0)(\C^n)$ is a generic linear section of $\cV$.  We choose $w_0 
\notin W$.  Then, by the inverse function theorem, we may change 
coordinates in $\C^n, 0$ so that $f_0$ has the form 
$$  f_0(x, y) \,\, =  \,\, \sum_{i = 1}^{N-1} x_i w_i \, + \,  g(x, y)\, w_0 $$
where $(x, y) = (x_1, \dots , x_{N-1}, y_1, \dots , y_{n-N+1})$, $\{ w_1, 
\dots w_{N-1}\}$ is a basis for $W$, and $dg(0) = 0$.

Then, $W$ being generic means that $f_1(x) = \sum_{i = 1}^{N-1} x_i w_i$ 
is a generic linear section.  Hence, by \fullref{Prophighmult} 
$\mu_{\cV}(f_1) = \mu_{N-1}(\cV) = 1$.  Then, let $\zeta_1, \dots , 
\zeta_{N-1}$ be the generators for $\dlog (H)$ for $H$ a  good defining 
equation for $\cV$.  In terms of the basis $\{w_i\}$, we write $\zeta_j = 
a^{(j)}_0 w_0 + \zeta_j^{\prime}$.  Then, the projection of $\cO_{\C^{N-1}, 
0}\{ w_0, w_1, \dots , w_{N-1}\}$ onto $\cO_{\C^{N-1}, 0}\{ w_0\} 
\simeq \cO_{\C^{N-1}, 0}$ along $\cO_{\C^{N-1}, 0}\{ w_1, \dots , w_{N-
1}\}$ induces an isomorphism
\begin{equation}
\label{Eqn9.1}
 N\cK_{H, e}\cdot f_1 \,\, \simeq  \,\, \cO_{\C^{N-1}, 0}/(a^{(1)}_0\circ 
f_1, \dots , a^{(N-1)}_0\circ f_1)  \, .
\end{equation}
However, by \fullref{ThmAFD} and the above, this has dimension $1$.  
Hence, $(a^{(1)}_0\circ f_1, \dots , a^{(N-1)}_0\circ f_1)$ provides a 
system of local coordinates for $\C^{N-1}, 0$.

For a $H$--holonomic linear free divisor $\cV$, germs which are 
transverse to $\cV$ off $0$ have finite $\cK_H$--codimension by 
\fullref{Rem2.1}.  Then, we may further apply a coordinate change and using 
Mather's Lemma to a homotopy from $f_0$ to conclude that the generic 
corank $1$ germs of finite $\cK_H$--codimension are 
$\cK_H$--equivalent to a germ of the form 
\begin{equation}
\label{Eqn9.1a}
  f_0(x, y) \,\, = \,\, \sum_{i = 1}^{N-1} x_i w_i \, + \,  g(y)\, w_0 
\end{equation} 
with $g(y)$ defining an isolated singularity on $\C^{n-N+1}, 0$.
We can then compute the singular Milnor number for generic corank $1$ 
germs as follows.
\begin{Proposition}
\label{Prop9.2}
Let $\cV \subset \C^N, 0$ be an $H$--holonomic linear free divisor, and 
$f_0(x, y)$ be a generic corank $1$ mapping of finite 
$\cK_H$--codimension for $\cV$, given by \eqref{Eqn9.1a}.  Then, 
$$ \mu_{\cV}(f_0) \,\, = \,\, \mu(g) \, .$$
\end{Proposition}
\begin{proof}
We note that $\pd{f_0}{x_j} = w_j$, and $\pd{f_0}{y_i} = \pd{g}{y_i}$.  In 
addition, by the above discussion, 
$$(a^{(1)}_0\circ f_0, \dots , a^{(N-1)}_0\circ f_0)  \,\, \equiv \,\, 
(a^{(1)}_0\circ f_1, \dots , a^{(N-1)}_0\circ f_1) \quad \mod (y_1, \dots 
y_{n-N+1}) , $$
so $(a^{(1)}_0\circ f_1, \dots , a^{(N-1)}_0\circ f_1, y_1, \dots y_{n-
N+1})$ form a system of local coordinates for $\C^N$.  Hence, as earlier, 
projecting $\cO_{\C^{n}, 0}\{ w_0, w_1, \dots , w_{N-1}\}$ onto 
$\cO_{\C^{n}, 0}\{ w_0\} \simeq \cO_{\C^{n}, 0}$ along $\cO_{\C^{n}, 0}\{ 
w_1, \dots , w_{N-1}\}$ induces an isomorphism
\begin{align}
\label{Eqn9.3}
 N\cK_{H, e}\cdot f_0 \,\, &\simeq  \,\, \left.\cO_{\C^{n}, 0}\middle/\left(a^{(1)}_0\circ 
f_0, \dots , a^{(N-1)}_0\circ f_0, \pd{g}{y_1}, \dots , \pd{g}{y_{n-N+1}}\right)\right. 
\notag  \\
&\simeq  \,\, \left.\cO_{\C^{n-N+1}, 0}\middle/\left(\pd{g}{y_1}, \dots , \pd{g}{y_{n-N+1}}\right)\right. .
\end{align}
Then, by \fullref{ThmAFD} and \eqref{Eqn9.3}, 
$$ \mu_{\cV}(f_0) \,\, =  \,\, \dim_{\C} N\cK_{H, e}\cdot f_0 \,\, =  \,\, 
\mu(g)  .\proved$$
\end{proof}
\begin{Remark}
\label{Rem9.2}
The above proof can be modified to apply to any $H$--holonomic free 
divisor $\cV \subset \C^N, 0$, and then $\mu(g)$ will be multiplied by 
$\mu_{N-1}(\cV)$.  
\end{Remark}
\subsection{A $\mu = \tau$-type Formula for Matrix 
singularities}
\label{Exam5.2}
We now consider a generic corank $1$ germ $f_0 \co \C^{n + N-1}, 0 \to M, 0$ 
where $M$ is any of the spaces of $m\times m$ matrices with ($\dim M = 
N$).  In the case $M = \Sym_m$, Bruce \cite{Br} shows that  $f_0$ is 
$\cK_M$--equivalent to germs of one of two types.  The first of which is 
generic in our sense
$$  f_0(x_1, \dots , x_{N-1}, y_1, \dots , y_n) \,\, = \,\,  \begin{pmatrix}
g_0(x, y) & x_1 & x_2 & \cdots & x_{m-1} \\ x_1  & x_m &  x_{m+1} & 
\cdots  & x_{2m-3}  \\ 
\cdots  & \cdots  & \cdots  & \cdots  & \cdots  \\
x_{m-1}  & x_{2m-3}  & \cdots  &\cdots  & x_{N-1} 
\end{pmatrix}  \, , $$
where $g_0(x, y) = \sum \gevar_i\, x_i + g(y_1, \dots , y_n)$ for generic 
tuples  $(\gevar_1, \dots , \gevar_{N-1})$, and $g$ defines an isolated 
hypersurface singularity on $\C^n$.  In fact, further normalization allows 
many $\gevar_i = 0$ (see \cite{Br}).  We will change coordinates so that 
the term $g_0(x, y)$ is in the lower right-hand corner to make use of the 
specific form of \eqref{Eqn5.7} in \fullref{Thmsym3} and the vector 
fields used to obtain the defining equation for $\cE^{sy}_3$.

For general and skew-symmetric cases there are analogous normal 
forms.  For example, for $2 \times 2$ general and  $4 \times 4$ 
skew-symmetric cases they take the form
$$ \begin{pmatrix}
  x_1 & x_2 \\ x_3 & g_0(x, y)  \end{pmatrix} 
\qquad \makebox{ and } \qquad
\begin{pmatrix}
0   & x_1 & x_2 & x_3 \\
-x_1&0  &  x_4  &  x_5 \\
-x_2  & -x_4  & 0 & g_0(x, y) \\
-x_3  & -x_5  & - g_0(x, y) & 0
\end{pmatrix} , $$ 
with $g_0(x, y)$ of the same form as above.  

Then, for this class of germs for any of the matrix types we obtain a $\mu 
= \tau$-type result.
\begin{Thm}[$\mu = \tau$ for generic corank 1 germs]
\label{Thm9.2}
We let $(\cD, \cE)$ denote any of the pairs $(\cD_2^{sy}, \cE_2^{sy})$, 
$(\cD_3^{sy}, \cE_3^{sy})$, $(\cD_2, \cE_2)$, or $(\cD_4^{sk}, \cE_4^{sk})$ 
and $f_0$ any of the corresponding generic corank $1$ germs as above.  
Then,
$$  \mu_{\cD}(f_0) \,\, = \,\, \mu(g)  \,\, = \,\, 
\cK_{H, e}\textrm{--}\codim (f_0)$$
where $H$ is the defining equation for the free divisor $\cE$.

If moreover g is weighted homogeneous, then
$$  \mu_{\cD}(f_0) \,\, = \,\, \cK_{H^{\prime}, e}\textrm{--}\codim (f_0)  
\,\, = \,\, \cK_{M, e}\textrm{--}\codim (f_0)  $$
where $H^{\prime}$ is the defining equation for $\cD$. 
\end{Thm}

\begin{proof}
We first consider $2 \times 2$ symmetric matrices.  
By \fullref{Thmsym2}, \fullref{ThmAFD} and generic reduction, 
$$\mu_{\cD_2^{sy}}(f_0)   \,\, = \,\,  \mu_{\cE_2^{sy}}(f_0)  \,\, = \,\, 
\cK_{H, e}\textrm{--}\codim (f_0) $$
where $H$ is the defining equation for $\cE_2^{sy}$.  Then a direct 
calculation analogous to that in the proof of \fullref{CorJacfor} 
shows $N \cK_{H, e}(f_0) \simeq \cO_{\C^n, 0}/\Jac(g)$, yielding the first 
equality.  Lastly, if $g$ is weighted homogeneous, with $H^{\prime}$ the 
defining equation for $\cD_2^{sy}$, then $\dlog (H^{\prime})$ has linear 
generators. Hence, for $\xi \in\dlog (H^{\prime})$,
$$ \xi \circ f_0 \,\, \in \,\, (x_1, x_2, g)\cdot \theta(f_0) \,\, \subset 
\,\, T\cK_{H, e}(f_0)  \, .$$
Hence, $\cK_{H^{\prime}, e}$--$\codim (f_0) = \cK_{H, e}$--$\codim (f_0)$, 
and  by \eqref{Eqn1.5} these equal $\cK_{M, e}$--$\codim (f_0)$, completing 
the proof.

The proof for $2\times 2$ general matrices is virtually identical to that 
for $2 \times 2$ symmetric matrices using instead 
\fullref{Thmgen2}.  

Next, for $3 \times 3$ symmetric matrices the argument is similar to 
that for the $2 \times 2$ case except for the first step.  Instead, we first, 
apply \fullref{Thmsym3} and generic reduction.  Since 
$df_0(0)(\C^{n+5})$ projects submersively onto all subspaces of 
dimension $\leq 5$, all terms of defining codimension $\geq 1$ are zero 
so we obtain
$$\mu_{\cD_3^{sy}}(f_0)   \,\, = \,\,  \mu_{\cE_3^{sy}}(f_0)  - \mu_{a, 
\cQ{a}}(f_0)  \, .$$
Then, by the meta-version of \fullref{Prop5.5} and generic 
reduction,
$$\mu_{a, \cQ{a}}(f_0)   \,\, = \,\,  \mu_{a, bd\cdot \cQ{a}}(f_0)  - 
\mu_{a, d, bc(bf-2ce)}(f_0) \, . $$
However, both $V(bd\cdot \cQ{a})$ and $V(bc(bf-2ce))$ are 
$H$--holonomic linear free divisors (by \fullref{thm:towerhholo} and  
\fullref{Prop3.4}).  By a change of coordinates in the source, we 
may assume that both $a$ and $d$ are coordinates for $\C^n$.  Thus, by 
\fullref{Prop9.2} applied to the restrictions of $f_0$ to the linear 
subspaces $V(a)$ and $V(a, d)$, 
$$ \mu_{a, bd\, \cQ{a}}(f_0)  \,\, = \,\,  \mu_{a, d, bc(bf-2ce)}(f_0) \,\, = 
\,\, \mu(g) \, . $$ 
Thus, $\mu_{a, \cQ{a}}(f_0)  = 0$ and $\mu_{\cD_3^{sy}}(f_0) =  
\mu_{\cE_3^{sy}}(f_0)$.  The remainder of the proof follows as for the $2 
\times 2$ symmetric case. 

Lastly, the proof for the $4 \times 4$ skew-symmetric case follows the 
proof for the $3 \times 3$ symmetric matrices, but with just one 
difference.  By \fullref{Thmsk4} and generic reduction, 
\eqref{Eqn8.1b} simplifies to 
\begin{equation}
\label{Eqn9.6}
 \mu_{\cD_4^{sk}}(f_0) \,\,  =  \,\,  \mu_{\cE_4^{sk}}(f_0) \, - \,  \mu_{a, 
f, (be-cd)}(f_0)  \, . \end{equation} 
The homogeneous generators $\zeta_i$ for $\dlog (H)$, with $H$ the 
defining equation for $\cE_4^{sk}$, consist of four linear vector fields and 
a quadratic vector field obtained from the Pfaffian vector field.  Thus, the 
$\pd{ }{x_{1,2}}$--components $a^{(j)}_0$ of the $\zeta_j\circ f_0$ have 
degrees $1, 1, 1, 1, 2$ in the $x_i$.  The first four give independent local 
coordinates, which we assume are $x_i$ for $i = 1, \dots , 4$.  The fifth 
term is obtained from the Pfaffian vector field; and modulo the ideal 
$(x_1, \dots , x_4)$, it is quadratic in $x_5$, $q(x_5, y)$, with 
coefficients in $y$.  Also, the $\pd{f_0}{y_i} = \pd{g}{y_i} w_0$ give the 
generators of  $\Jac(g) \{w_0\}$.  Thus, by a calculation similar to the 
above one for $3 \times 3$ symmetric matrices together with 
\fullref{ThmAFD} (also see \fullref{Rem9.2})
 $$\mu_{\cE_4^{sk}}(f_0) \,\,  =  \,\, \cK_{H, e}\textrm{--}\codim(f_0)  
\,\,  =  \,\,  2\mu(g) \, . $$
However, by \fullref{Thmgen2}, generic reduction and 
\fullref{Prop9.2} applied to the restriction of $f_0$ to $V(a, f)$, 
$$\mu_{a, f, (be-cd)}(f_0) \,\, = \,\, \mu_{a, f, bc(be-cd)} (f_0)   \,\, = \,\,  
\mu(g)  \, . $$ 
Hence, we obtain from \eqref{Eqn8.6} and \eqref{Eqn9.6}
$$  \mu_{\cD_4^{sk}}(f_0) \,\,  =  \,\,  \mu(g) \, . $$
The remainder of the proof is analogous to that for $3 \times 3$ 
symmetric matrices.  
\end{proof}

\begin{Remark}
\label{Rem9.4}
What is surprising in all of these cases is that the number of singular 
vanishing cycles for the matrix singularities equals the number of 
vanishing cycles for the isolated singularity $g$, although there is at this 
point no known geometric reason for this agreement.
This leads to:

{\bf Conjecture}\stdspace
{\it
For all generic corank $1$ matrix singularities for $m 
\times m$ symmetric, general, or skew-symmetric (for $m$ even) 
matrices, there is a $\mu = \tau$ result, where $\mu = \mu_{\cD}$ and 
$\tau = \cK_{H, e}$--$\codim$, for $H$ the defining equation for the 
appropriate $\cE$.  If moreover $g$ is weighted homogeneous, both of 
these equal $\cK_{M, e}$--$\codim = \cK_{H^{\prime}, e}$--$\codim = 
\mu(g)$, where $H^{\prime}$ is the defining equation for $\cD$. 
}
\end{Remark}

This result contrasts with  the situation for generic corank $1$ germs 
$f_0 \co \C^n, 0 \to M_{2,3}, 0$ for the varieties $V_{i, j}$ in the space of $2 
\times 3$ general matrices.  Now by \fullref{Thmgen2.3} and  
generic reduction the singular Milnor number is zero.  Then, using generic 
reduction and \fullref{ThmVanchi} together with 
\fullref{Prop9.2}, we obtain the following for the variety of singular matrices 
$\cV$ in $M_{2, 3}$.  
\begin{Corollary}
\label{Cor9.3}
If $f_0 \co \C^{n}, 0 \to M_{2, 3}$ is a generic corank $1$ germ as above with 
$n \geq 6$, then 
$$  \tilde \chi_{\cV}(f_0) \,\, = \,\, (-1)^{n-1}\mu_{V_{1 2 3}}(f_0) = 
(-1)^{n-1}\mu(g)  \, . $$
If $g$ is weighted homogeneous, these equal the 
$\cK_{M, e}$--codimension of $f_0$.  
\end{Corollary}

\fullref{Cor9.3} substitutes for the $\mu = \tau$ formula in this 
case.  A simple example of this can be seen in the list in
\cite[Theorem 3.6]{FN} for 
codimension $2$ Cohen--Macaulay singularities in $\C^6$.  Example 
$\Omega_k$ in the list, has $g(u) = u^k$, an $A_{k-1}$ singularity and 
the $\tau$, which is the $\cK_{M, e}$--codimension, equals $k-1$.   
Calculations of the singular vanishing Euler characteristic using the  
Macaulay2 package \cite{P2} for computing the formula in 
\fullref{ThmVanchi} yields $-(k-1)$ as claimed above.  
\subsection{$\mu = \tau -1$-type Results for $2\times 3$ Cohen--Macaulay Surface 
Singularities}
Having obtained above a number of $\mu=\tau$ results for hypersurfaces, 
we ask what form results take for Cohen--Macaulay singularities defined 
as $2 \times 3$ matrix singularities.  If $f_0\co \C^4,0\to M_{2,3},0$ is a 
germ transverse off $0$ to the variety $\cV$ of singular matrices, then  
$\cV_0=f_0^{-1}(\cV)$ is an isolated Cohen--Macaulay surface singularity.  
We use the $\cK_{M,e}$--codimension of $f_0$ for $\tau$, and
the Milnor number $\mu(\cV_0)$ for $\mu$.

Specifically the simple isolated Cohen--Macaulay surface singularities 
arise in this way and were classified by Fr\"{u}hbis-Kr\"{u}ger and Neumer 
(\cite[Theorem 3.3]{FN}). These 
turn out to be precisely the rational triple points (c.f. \cite{Tj}).  They 
include both a number of infinite families and discrete cases.  As well in 
\cite{FN} are identified the singularities just outside the simple range.

Until recently the only method to compute the Milnor number involved 
using a partial resolution of $\cV_0$.  There are now two new ways to 
compute the Milnor number.  In the recent thesis of Pereira 
(\cite{Pereira}), she applies a L\^{e}--Greuel type method to a generic 
linear function on the surface.  This method requires that the number of 
critical points of the linear function on the Milnor fiber be computed 
directly by hand.  Also, \fullref{ThmCM} provides an effective 
formula for computing $\mu(\cV_0)$, and this has been implemented by 
the second author as a package \cite{P2} in Macaulay2.  Taken together, 
these computations include all of the simple isolated
Cohen--Macaulay surface singularities, as well as certain non-simple cases. 

\vspace{1ex}
{\it Summary of the Results for Isolated $2\times 3$ Cohen--Macaulay Surface 
Singularities: } \hfill
\begin{itemize}
\item[1)] Pereira computes the Milnor number for many discrete cases and 
the entirety of many of the infinite families of simple singularities.  
Based on her results she has conjectured (6.3.1 of \cite{Pereira}) and 
verified for her cases that for $\cV_0$ quasihomogeneous,
\begin{equation}
\label{cmmutau}
\mu(\cV_0)=\tau(\cV_0)-1.
\end{equation}
\item[2)] Using the Macaulay2 package \cite{P2}, we have verified 
\eqref{cmmutau} for all of the discrete examples, for the first few 
examples of each infinite family, and for a number of cases just outside 
the simple region (e.g., \fullref{tableinc4} in the Appendix \S 
\ref{S:Appen}).
\end{itemize}
With further work, \fullref{ThmCM} should provide a method to prove 
\eqref{cmmutau} for large classes of singularities.  One immediate 
consequence is that while for ICIS $\gg = \tau - \mu \leq 0$, now for 
non-ICIS $\gg = \tau - \mu$ becomes positive.   The relation 
\eqref{cmmutau} would be a striking complement to a similar pattern 
found in listings of certain space curve singularities (see Tables 1, 2a, 2b 
of \cite{Fr}).
\subsection{$\mu = \tau -\gg$ for $2\times 3$ Cohen--Macaulay 3--fold 
Singularities in $\C^5$}
We next consider isolated Cohen--Macaulay 3--fold
singularities $\cV_0, 0 \subset \C^5, 0$ defined by  $f_0\co \C^5,0\to 
M_{2,3},0$, with $\cV_0=f_0^{-1}(\cV)$.  Again by the results of 
Greuel--Steenbrink \cite{GS}, the first (vanishing) Betti number of the 
Milnor fiber of $\cV_0$, $b_1(\cV_0) = 0$.  As there are two possibly 
non-vanishing Betti numbers for the Milnor fiber, we replace the Milnor 
number by $b_3(\cV_0) -b_2(\cV_0)$.   We can use 
\fullref{Thmb2-b3} to compute $b_3(\cV_0) - b_2(\cV_0)$ and investigate 
whether an analog of \eqref{cmmutau} holds.

We apply \fullref{Thmb2-b3} to the classification of simple isolated
Cohen--Macaulay 3--fold singularities in $\C^5$
(\cite[Theorem 3.5]{FN}). We compute \eqref{Eqn7.5c} using the Macaulay2 package 
\cite{P2}, and summarize the results in \fullref{tableinc5} in the 
Appendix \S\ref{S:Appen}.

We summarize the main observed conclusions from the calculations.  
These conclusions concern the values and behavior of $\gg = \tau - (b_3 - 
b_2)$ (where $\tau = \cK_{M, e}$--$\codim$), and the behavior of $\gg$ and 
$b_3 -b_2$ in simple infinite families.  We emphasize that although we 
state the expected form of these for infinite families, we have so far only 
verified them for a small range of values in each infinite family.

\vspace{1ex}
{\it Summary of the Results for Isolated $2\times 3$ Cohen--Macaulay 3--fold 
Singularities : } \hfill
\begin{itemize}
\item[a)] $\gg \geq 2$ and increases in value as we move higher in the 
classification.
\item[b)]  $b_3 - b_2 \geq -1$, with equality for the generic linear 
section and one infinite family. 
\item[c)]  $b_3 - b_2$ is constant for certain infinite families with 
values 
$-1$ (one family), $0$ (two families), and $1$ (two families).
\item[d)]  $\gg$ is constant in all other considered infinite families in 
\fullref{tableinc5}  with only one exception where both $b_3 - b_2$ 
and $\gg$ increase with $\tau$.
\item[e)]  For singularities of the form 
$\begin{pmatrix}
x & y & z \\
w & v & g(x, y)
\end{pmatrix}$    
with $g$ a simple hypersurface singularity (cases 2--6 in 
\fullref{tableinc5}), $\gg = 3$ and $b_3 - b_2 = \mu(g) - 1$.
\end{itemize}

As each $b_i\geq 0$, knowing $b_3 - b_2$ gives 
lower bounds on $b_3$ when $b_3 - b_2 >0$, and 
on $b_2$ when $b_3 - b_2 <0$.  In particular, the generic Cohen--Macaulay 
3--fold singularity as well as one infinite family must have $b_2 > 0$.  
In fact, we expect that both $b_2$ and $b_3$ will increase with $\tau$ in 
families with $b_3 - b_2$ constant.  
\begin{Remark}
These results reveal that there are (at least) two quite different (and 
mutually exclusive) types of behavior occurring for infinite families of 
isolated Cohen--Macaulay $3$ fold singularities: one where 
$b_3 - b_2$ is constant in the family and one where $\gg$ is constant.  A 
basic question is what different geometric properties are responsible for 
the two different types of behavior?  Second, as $\gg$ increases within 
the classification, how can it be computed independently via other 
geometric properties of the singularities?
\end{Remark}

\section{Appendix: Computations for $2\times 3$ Cohen--Macaulay Singularities}  
\label{S:Appen} 
\begin{longtable}{lll}
\caption{Some non-simple isolated $2\times 3$ Cohen--Macaulay surface singularities 
in $\C^4$, from the proof of \cite[Theorem 3.3]{FN}.} \\
Presentation matrix & $\tau$ & $\mu$\\
\hline
\endhead
$\begin{pmatrix}
z & y & x \\
x & w & z^2+y^4
\end{pmatrix}$ & 11 & 10 \\
$\begin{pmatrix}
z & y & x \\
x & w & y^3+z^3
\end{pmatrix}$ & 10 & 9 \\
$\begin{pmatrix}
z & y & x^2+y^2 \\
x & w & w^2+xw+z^2
\end{pmatrix}$ & 13 & 12 \\
$\begin{pmatrix}
x & y & z \\
w & zx+x^2 & w+yz
\end{pmatrix}$ & 9 & 8 \\
$\begin{pmatrix}
z & y & x^2 \\
w^2 & x & y+w^2
\end{pmatrix}$ & 8 & 7 \\
$\begin{pmatrix}
z & y & x^2+z^2 \\
w^2 & x & y+w^2
\end{pmatrix}$ & 8 & 7
\label{tableinc4} 
\end{longtable}

\begin{longtable}{lcll}
\caption{The simple isolated $2\times 3$ Cohen--Macaulay 3--fold singularities in
$\C^5$, from \cite[Theorem 3.5]{FN}.} \\
Presentation matrix & Parameters computed & $\tau$ & $b_3 - b_2$ \\
\hline
\endfirsthead

\multicolumn{4}{c}
{(Table \thetable{}, continued)}
\\
Presentation matrix & Parameters computed & $\tau$ & $b_3 - b_2$ \\
\hline
\endhead

\multicolumn{4}{c}
{(Table continues)}
\\
\endfoot

\endlastfoot

$\begin{pmatrix} x & y & z \\ w & v & x \end{pmatrix}$
&   &  $1$ & $-1$ \\ 
$\begin{pmatrix} x & y & z \\ w & v & x^{k+1}+y^2 \end{pmatrix}$
& $1\leq k\leq 4$  & $k+2$  & $k-1$ \\ 
$\begin{pmatrix} x & y & z \\ w & v & xy^2+x^{k-1} \end{pmatrix}$
& $4\leq k\leq 6$  & $k+2$ & $k-1$ \\ 
$\begin{pmatrix} x & y & z \\ w & v & x^3+y^4 \end{pmatrix}$
&   & $8$  & $5$ \\ 
$\begin{pmatrix} x & y & z \\ w & v & x^3+xy^3 \end{pmatrix}$
&   & $9$  & $6$ \\ 
$\begin{pmatrix} x & y & z \\ w & v & x^3+y^5 \end{pmatrix}$
&   & $10$  & $7$ \\ 
$\begin{pmatrix} w & y & x \\ z & w & y+v^k \end{pmatrix}$
&  $2\leq k\leq 5$ & $2k-1$  & $-1$ \\ 
$\begin{pmatrix} w & y & x \\ z & w & y^k+v^2 \end{pmatrix}$
&  $2\leq k\leq 5$ &  $k+2$ & $k-2$ \\ 
$\begin{pmatrix} w & y & x \\ z & w & yv+v^k \end{pmatrix}$
&  $2\leq k\leq 5$ & $2k$  & $0$ \\ 
$\begin{pmatrix} w+v^k & y & x \\ z & w & yv \end{pmatrix}$
&  $2\leq k\leq 5$ & $2k+1$  & $0$ \\ 
$\begin{pmatrix} w+v^2 & y & x \\ z & w & y^2+v^k \end{pmatrix}$
&  $2\leq k\leq 5$ &  $2k$ & $k-2$ \\ 
$\begin{pmatrix} w & y & x \\ z & w & y^2+v^3 \end{pmatrix}$
&  & $7$  & $1$ \\ 
$\begin{pmatrix} v^2+w^k & y & x \\ z & w & v^2+y^l \end{pmatrix}$
& $2 \leq k \leq l \leq 6$ &  $k+l+1$ & $k+l-3$ \\ 
$\begin{pmatrix} v^2+w^k & y & x \\ z & w & yv \end{pmatrix}$
& $2\leq k\leq 5$ &  $k+4$ & $k-1$ \\ 
$\begin{pmatrix} v^2+w^k & y & x \\ z & w & y^2+v^l \end{pmatrix}$
& $2 \leq k \leq 3;\, 3 \leq l \leq 7$ &  $k+l+2$ & $k+l-3$\\ 
$\begin{pmatrix} wv+v^k & y & x \\ z & w & yv+v^k \end{pmatrix}$
& $3\leq k\leq 6$ &  $2k+1$ & $1$ \\ 
$\begin{pmatrix} wv+v^k & y & x \\ z & w & yv \end{pmatrix}$
& $3\leq k\leq 6$ &  $2k+2$ & $1$ \\ 
$\begin{pmatrix} wv+v^3 & y & x \\ z & w & y^2+v^3 \end{pmatrix}$
&  &  $8$ & $2$ \\ 
$\begin{pmatrix} wv & y & x \\ z & w & y^2+v^3 \end{pmatrix}$
&  &  $9$ & $2$ \\ 
$\begin{pmatrix} w^2+v^3 & y & x \\ z & w & y^2+v^3 \end{pmatrix}$
&  &  $9$ & $3$ \\ 
$\begin{pmatrix} z & y & x \\ x & w & v^2+y^2+z^k \end{pmatrix}$
& $2\leq k\leq 5 $ &  $k+4$ & $k$ \\ 
$\begin{pmatrix} z & y & x \\ x & w & v^2+yz+y^k w \end{pmatrix}$
& $1\leq k\leq 4$ &  $2k+5$ & $2k+1$ \\ 
$\begin{pmatrix} z & y & x \\ x & w & v^2+yz+y^{k+1} \end{pmatrix}$
& $2\leq k\leq 5$ &  $2k+4$ & $2k$ \\ 
$\begin{pmatrix} z & y & x \\ x & w & v^2+yw+z^2 \end{pmatrix}$
&  &  $8$ &  $4$ \\ 
$\begin{pmatrix} z & y & x \\ x & w & v^2+y^3+z^2 \end{pmatrix}$
&  &  $9$ &  $5$ \\ 
$\begin{pmatrix} z & y & x+v^2 \\ x & w & vy+z^2 \end{pmatrix}$
&  &  $7$ &  $2$\\ 
$\begin{pmatrix} z & y & x+v^2 \\ x & w & vz+y^2 \end{pmatrix}$
&  &  $8$ &  $3$\\ 
$\begin{pmatrix} z & y & x+v^2 \\ x & w & y^2+z^2 \end{pmatrix}$
&  &  $9$ &  $4$
\label{tableinc5} 
\end{longtable} 
\end{document}